\documentclass[]{amsart}
\usepackage{amsmath}
\usepackage{amssymb}
\usepackage{amsthm}
\usepackage{amsfonts}
\usepackage[abbrev]{amsrefs}
\usepackage{bm}
\usepackage{graphicx}
\usepackage{enumerate}
\usepackage{url}
\usepackage{epsfig}
\usepackage{color}
\usepackage{float}
\usepackage{setspace}
\usepackage{comment}
\usepackage{appendix}
\usepackage[hang,small,bf]{caption}
\usepackage[subrefformat=parens]{subcaption}
\theoremstyle{definition}

\newtheorem{thm}{Theorem}[section]

\newtheorem{lem}[thm]{Lemma}

\newtheorem{prop}[thm]{Proposition}

\newtheorem{ex}[thm]{Example}

\newtheorem{defi}[thm]{Definition}
  
\theoremstyle{remark}
\newtheorem{rem}{Remark}[section]

\newtheorem{conj}[thm]{Conjecture}
\newtheorem*{acknowledgment}{Acknowledgments}

\numberwithin{equation}{section}\numberwithin{figure}{section}
\def\co{\colon\thinspace}

\newcommand{\bigslant}[2]{{\raisebox{.2em}{$#1$}\left/\raisebox{-.2em}{$#2$}\right.}} 
\newcommand{\HD}{\mathcal{H}}

\title{quantum invariants of closed framed $3$-manifolds based on ideal triangulations}

\author{Serban Matei Mihalache}%
\address{Department of Mathematics, Tohoku University, 
6-3, Aoba, Aramaki-aza, Aoba-ku, 
Sendai, 980-8578, Japan}
\email{matei.mihalache.q3@dc.tohoku.ac.jp}

\author{Sakie Suzuki}%
\address{Department of Mathematical and Computing Science, School of Computing,
Tokyo Institute of Technology,
2-12-1 Ookayama, Meguro-ku, Tokyo 152-8552, Japan}
\email{sakie@c.titech.ac.jp}

\author{Yuji Terashima}%
\address{Department of Mathematics, Tohoku University, 
6-3, Aoba, Aramaki-aza, Aoba-ku, 
Sendai, 980-8578, Japan}
\email{yujiterashima@tohoku.ac.jp}

\begin{document}

\maketitle

\begin{abstract}
We construct a new type of quantum invariant of closed framed $3$-manifolds with the vanishing first Betti number. The invariant is defined for any finite dimensional Hopf algebra, such as small quantum groups, and is based on ideal triangulations.  
We use the canonical element of the Heisenberg double, which satisfies a pentagon equation, and graphical representations  of $3$-manifolds introduced by R. Benedetti and C. Petronio. The construction is simple and easy to be understood intuitively; the pentagon equation reflects the Pachner $(2,3)$ move of ideal triangulations and the non-involutiveness of the Hopf algebra reflects framings.
For an involutory Hopf algebra, the invariant reduces to an invariant of closed combed $3$-manifolds.
For an involutory unimodular counimodular Hopf algebra, the invariant reduces to the topological invariant of  closed $3$-manifolds which is introduced in our previous paper. 
In this paper we formalize the construction using more generally a Hopf monoid in a symmetric pivotal category and use tensor networks for calculations.
\end{abstract}

\tableofcontents

\section{Introduction}
In his foundational work \cite{W}, E. Witten observed that the Chern-Simons theory leads to a new type of $3$-manifold invariants. After that N. Reshetikhin and V. Turaev \cite{RT} discovered a rigorous construction of invariants which is believed to be a mathematical realization of Chern-Simons-Witten theory.  
The Witten-Reshetikhin-Turaev (WRT) invariant is constructed based on  link surgery presentations of $3$-manifolds, and algebraic ingredients are finite dimensional representations of quantum groups, or, more generally, of ribbon Hopf algebras.
In the same period, various constructions of quantum invariants of 3-manifolds were realized: 
the Hennings-Kauffman-Radford invariant \cites{H, KR}  by  link surgery presentations and integrals of ribbon Hopf algebras, 
the Turaev-Viro
invariant \cites{TV, BW} by state sum models based on triangulations and quantum $6j$-symbols,  the Kuperberg invariant \cites{Ku1, Ku2} by  Heegaard decompositions and finite dimensional Hopf algebras with theory of their integrals. 
Although the above invariants are based on different decompositions of $3$-manifolds and their algebraic ingredients are slightly different, they are closely related at least when we use quantum groups at roots of unity.

This paper presents a new construction of quantum invariants of closed framed $3$-manifolds with the vanishing first Betti number.  
The construction is based on ideal triangulations and finite dimensional Hopf algebras. Compared to the above invariants, the definition is quite simple as we will see below, where
we use only the structure constants of a Hopf algebra without using any representation theory or integral.
The invariant can be formulated as a monoidal functor, and for the well-definedness of the invariant, each argument is local and finite like the Reidemeister moves for knot diagrams.
The key ingredient of the invariant is the canonical element $T\in \mathcal{H}(H)^{\otimes 2}$ of the Heisenberg double $\mathcal{H}(H)$ of a Hopf algebra $H$, satisfying the pentagon equation (cf. \cites{BS, Ka})
\begin{align*}
T_{12}T_{13}T_{23}=T_{23}T_{12} \in \mathcal{H}(H)^{\otimes 3},
\end{align*}
which, in our invariant,  will be the algebraic counterpart   of the Pachner $(2,3)$ move of ideal triangulations.
For the  quantum Borel subalgebra of $\mathfrak{sl}_2$, the pentagon equation of $T$ turns out to be essentially the Fadeev-Kashaev's pentagon identity for the quantum dilogarithm \cites{FK, Ka}.
The pentagon identity for the quantum dilogarithm was the crucial result sitting behind a sequence of important works \cites{Ka2, Ka3, Ka4, AK}  related to the Kashaev invariant of links, the volume conjecture, and investigation of $3$-manifolds invariants and TQFT using the quantum dilogarithm and quantum Teichm\"{u}ller Theory.

In previous works, we have been investigating constructions of quantum type 
invariants of $3$-manifolds  such that the Pachner $(2,3)$ move corresponds to the pentagon equation of $T$.
In \cite{S}, for an arbitrary finite dimensional Hopf algebra, the second author reconstructed the universal quantum invariant \cites{Law90, Oh} of framed tangles by associating $T^{\pm 1}$ with each ideal tetrahedron in  tangle complements. 
Here, the universal quantum invariant includes all informations on the Reshetkhin-Turaev invariant of links associated with the Drinfeld double of a Hopf algebra, in particular, the colored Jones polynomial and the Kashaev invariant when we use the small quantum Borel subalgebra of $\mathfrak{sl}_2$. 
In \cite{MST}, we obtained a topological invariant of closed $3$-manifolds by extending the construction in \cite{S} restricting Hopf algebras to be involutory, unimodular and counimodular.
Here, for an involutory Hopf algebra we got an invariant of closed combed $3$-manifolds, and in addition, we used unimodularity and counimodularity  to mod out combings to obtain a topological invariant of closed $3$-manifolds. 
Note that quantum groups are non-involutory Hopf algebras, so the invariants in \cite{MST} are not quite  ``quantum" ones.

In the present paper, in order to define $3$-manifolds invariants using arbitrary finite dimensional Hopf algebras, we make use of  framing structures of $3$-manifolds.
To represent closed framed $3$-manifold,  we basically use closed normal o-graphs introduced by R. Benedetti and C. Petronio \cite{BP}, which represent spines (the dual notion of ideal triangulations) of closed combed $3$-manifolds. Each vertex of a closed normal o-graph represents an ideal tetrahedron with a combing inside, and each connecting edge tells a pair of faces of tetrahedra which are glued to each other.
In order to represent framings, in \cite{BP} they assigned a $\mathbb{Z}/2\mathbb{Z}$ weight on each edge of closed normal o-graphs.
In this paper we instead assign an integer, and by modifying the argument in \cite{BP}, we show that framed integral normal o-graphs  represent closed framed $3$-manifolds when the first Betti number  vanishes (Proposition \ref{prop: framed 3-mfd BP-diagram}), see Figure \ref{EX1}  for examples with $S^3$ and $L(2,1)$ (see Example \ref{lens} for the details).\begin{figure}[ht]
    \centering
    \includegraphics{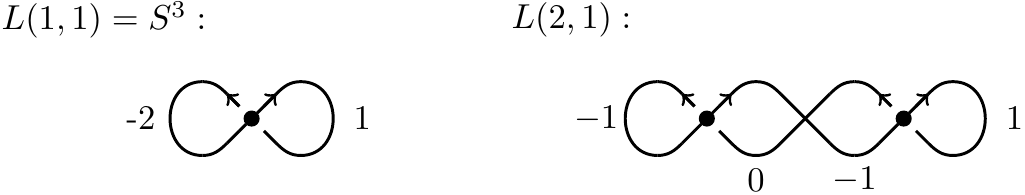}
    \caption{Framed integral normal o-graphs representing $S^3$ and $L(2,1)$ with certain framings.}\label{EX1}
\end{figure}
\noindent The modification to integer weights will be necessary when we use non-involutory Hopf algebras in the following step.

 Once we get the formulation representing closed framed $3$-manifolds, the construction of the invariant is very simple;
for a finite dimensional Hopf algebra $H$, we replace each vertex and each weighted edge of an integral normal o-graph with a tensor in $H$ as in Figure \ref{EX2} (see Section \ref{sec:Invariant} for the details), and then contract these tensors according to how the diagram was connected. See Figure \ref{EX3} for an example (see Example \ref{lens} and \ref{ex:invariant of S3} for the details of the calculation). 

\begin{figure}[h]
  \centering
  \includegraphics{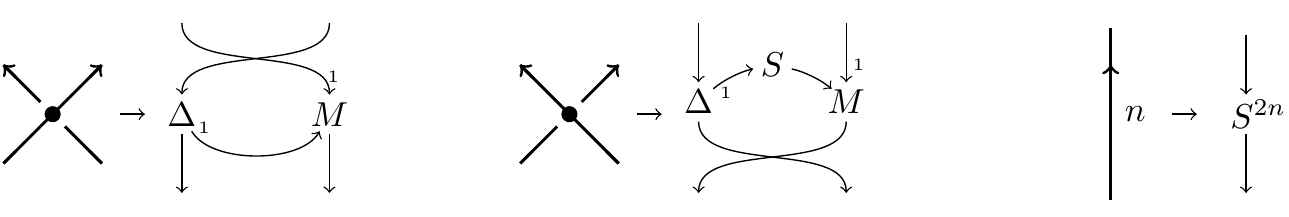}
  \caption{Definition of invariant.}\label{EX2}
\end{figure}
\begin{figure}[h]
    \centering
    \includegraphics{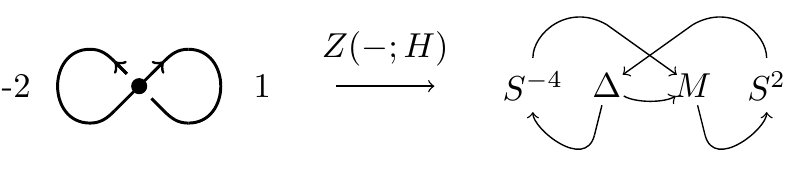}
    \caption{Calculation of the invariant for  framed $S^3$ in Figure \ref{EX1}.}\label{EX3}
\end{figure}

In the above construction, the tensor associated with each vertex comes from the canonical element $T^{\pm 1}$ of the Heisenberg double which satisfies the pentagon equation, see Appendix \ref{Heisenberg double construction} for the details.  Note that
the construction does not depends on integer weights when the Hopf algebra is involutory, i.e., when $S^2=\mathrm{id}$, and it reduces to the invariant defined in \cite{MST}.
The construction includes only the structure constants of Hopf algebras, and thus the invariant can be defined for each Hopf monoid in a symmetric pivotal category, where we can use not only finite dimensional Hopf algebras but also finite dimensional Hopf super algebras. 
We note that our construction behaves well as a monoidal functor because the equivalent moves for framed integral normal o-graphs are locally and finitely generated  (as those for normal o-graphs), unlike the constructions by link surgery presentations or Heegaard decompositions of $3$-manifolds  where the equivalent moves contain handle slides.

The invariant we define in this paper is for pairs $(M,f)$ of a closed oriented $3$-manifold $M$ and its framing $f$, which is a trivialization of tangent bundle $TM$. 
It is well known that every closed oriented $3$-manifold admits a framing.
If we would like to reduce the invariants to topological ones, theoretically, we have the following way using spin structures.
In \cite{KM2}, it is observed that, given a spin structure $s$ over $M$, one can always choose a canonical framing with underline spin structure $s$. This implies that our invariants can be boiled down to invariants of spin $3$-manifolds.
The cardinal of the set of framings is always infinite, while that of spin structures is always finite. Thus we can add the invariants over all  spin structures to get topological invariants.

It would be interesting to investigate relations with other quantum invariants of $3$-manifolds which we mentioned in the beginning.  In particular, the Kuperberg invariant for a non-involutory Hopf algebra is also an invariant of framed 3-manifolds.
Although both of his and our invariants use non-involutory Hopf algebras and framed 3-manifolds, the
treatments of framings are quite different and at the moment we do not know the exact relation between these two invariants. Recall that the WRT invariant and the Turaev-Viro invariant are constructed using representation categories of  Hopf algebras. Our invariant  does not use  representation categories, but as a result our invariant for the small quantum Borel subalgebra of $\mathfrak{sl}_2$ seems to be equal to the $SO(3)$ WRT invariant times the cardinal number of the first homology group up to  multiplication by $q$ (Conjecture \ref{vsWRT}).
  
As we mentioned earlier, our invariant is  constructed from the canonical element $T$ associated with a Hopf algebra  and when we use quantum Borel subalgebra  of $\mathfrak{sl}_2$, which is an infinite dimensional Hopf algebra, the resulting pentagon equation is essentially the Fadeev-Kashaev's pentagon identity for the quantum dilogarithm.
In our setting, the Hopf algebra must be finite dimensional and thus we need to use the small quantum Borel subalgebra.
It would be an interesting problem to see if our construction can be extended to infinite dimensional quantum Borel subalgebra, which will result in an invariant based on the quantum dilogarithm.

The paper is organized as follows.
In Section 2, we recall symmetric pivotal categories with trivial twists and Hopf monoids in them. In symmetric pivotal categories with trivial twists, we can draw their morphisms using a generalization of tensor networks of linear maps.
In Section 3, we define integral normal o-graphs and their local moves, and construct invariants of them using Hopf monoids.
In Section \ref{sec:Branched standard spine} we introduce framed integral normal o-graphs and  show Proposition \ref{prop: framed 3-mfd BP-diagram}, which says that if the first Betti number vanishes, closed framed 3-manifolds can be represented by framed integral normal o-graphs. 
In the last section, we give examples of invariants for $S^3$ and $L(2,1)$ with certain framings, and observe the similarity of the invariant with the $SO(3)$ WRT invariant when we use the small quantum Borel subalgebra of $\mathfrak{sl}_2$. In Appendix A, we  explain an alternative construction of the invariant using the Heisenberg double of a Hopf algebra. 
 
\begin{acknowledgment}
 We would like to thank K. Hikami, M. Ishikawa, A. Kato, Y. Koda for valuable discussions. This work is partially supported by
 JSPS KAKENHI Grant Number JP17K05243, JP19K14523, JP21K03240 and by JST CREST Grant Number JPMJCR14D6.
\end{acknowledgment}

\section{Hopf monoid in symmetric pivotal category}
\label{sec:Hopf algebras and tensor nertworks}
In this section, we recall the definition of Hopf monoids in symmetric pivotal categories. For example, finite dimensional Hopf algebras and finite dimensional Hopf super algebras are Hopf objects in the category $\text{Vect}^{\text{Fin}}_{\mathbb{K}}$ of finite dimensional vector spaces and in the category $\text{SVect}^{\text{Fin}}_{\mathbb{K}}$ of finite dimensional  super vector spaces, respectively.

\subsection{Symmetric pivotal category}
Throughout the paper, monoidal categories are assumed to be strict.

\begin{defi}[Left dual and left-rigid]Let $(\mathcal{C}, \otimes, \mathbb{I})$ be a monoidal category with the tensor product $\otimes$ and the unit object $\mathbb{I}$. 
For an object $V \in \mathcal{C}$, a \textit{\textbf{left dual}} of $V$ is a pair $(V^*, \text{ev}_V, \text{coev}_V)$ of an object $V^*\in \mathcal{C}$ and morphisms $\text{ev}_V\co V^*\otimes V\to \mathbb{I}$, $\text{coev}_V\co\mathbb{I}\to V\otimes V^*$ satisfying the following conditions:
\begin{align*}
  (\text{id}_V\otimes \text{ev}_V)\circ(\text{coev}_V\otimes \text{id}_V)&=\text{id}_V,\\
  (\text{ev}_V\otimes \text{id}_{V^*})\circ(\text{id}_{V^*}\otimes \text{coev}_V)&=\text{id}_{V^*}.
\end{align*}
A monoidal category is called \textit{\textbf{left-rigid}} if every object has a specified left dual. 
\end{defi}
For any left-rigid monoidal category $\mathcal{C}$, there exists a functor
\begin{align*}
  &?^*\co \mathcal{C}\to\mathcal{C}^{\text{op}}
\end{align*}
which sends an object $V$ to its dual $V^*$ and a morphism $f$ to $f^*:=  (\text{ev}_V\otimes \text{id}_{V^*})\circ (\text{id}_{V^*}\otimes f \otimes \text{id}_{V^*})\circ (\text{id}_{V^*}\otimes\text{coev}_V)$. 

\begin{defi}[Pivotal structure]
A left-rigid monoidal category $\mathcal{C}$ is called a \textit{\textbf{pivotal category}} if $\mathcal{C}$ is equipped with monoidal natural isomorphism 
\begin{align*}
  &\omega\co\text{id}_{\mathcal{C}}\to ?^{**}
\end{align*}
between identity functor $\text{id}_{\mathcal{C}}$ and $?^{**}$.
\end{defi}

We can define the notions of a \textit{\textbf{right dual}} and \textit{\textbf{right-rigid}} similarly.
A pivotal category $\mathcal{C}$ is always right-rigid with the right dual $(^{*}V, \widehat{\text{coev}}_V, \widehat{\text{ev}}_V)$ of $V\in \mathcal{C}$ as follows.
\begin{align*}
  ^{*}V&:=V^*,\\
  \widehat{\text{coev}}_V&:=(\text{id}_{V^*}\otimes\omega_{V^{**}}^{-1})\circ\text{coev}_{V^*},\\
  \widehat{\text{ev}}_V&:=\text{ev}_{V^*}\circ(\omega_V\otimes\text{id}_{V^*}).
\end{align*}

Recall that a braided monoidal category $\mathcal{C}$ with the braidings $\{c_{V,W}:V\otimes W\to W\otimes V\}_{V,W\in\mathcal{C}}$ is \textit{\textbf{symmetric}} if $c_{V,W}=c_{W,V}^{-1}$ for every $V,W\in\mathcal{C}$.

\begin{defi}[Trivial twists]\label{def:trivial twists}
  Let $\mathcal{C}$ be a symmetric pivotal category.
We say that $\mathcal{C}$ has \textit{\textbf{trivial twists}} if the following equality holds for every object $V$.
\begin{align*}
\text{id}_V&=(\text{id}_V\otimes\widehat{\text{ev}})\circ (c_{V,V}\otimes\text{id}_{V^*})\circ(\text{id}_V\otimes \text{coev}_V),\\
\text{id}_V&=(\text{ev}\otimes \text{id}_V)\circ (\text{id}_{V^*}\otimes c_{V,V})\circ(\widehat{\text{coev}}\otimes\text{id}_V).
\end{align*}
\end{defi}

\begin{ex}A left-rigid symmetric monoidal category $\mathcal{C}$ always admits a natural pivotal structure 
 $\omega_V:=(\text{ev}_{V}\otimes\text{id}_{V^{**}})\circ(\text{id}_{V^*}\otimes c_{V^*,V})\circ(\text{coev}_{V^*}\otimes\text{id}_V)$, for $V\in\mathcal{C}$, which has trivial twists.
\end{ex}

\begin{ex}
\label{ex:category of vector spaces}
Let $\mathbb{K}$ be a commutative ring  with unit, and $\text{Proj}^{\text{Fin}}_{\mathbb{K}}$  the category of finitely generated projective modules over $\mathbb{K}$.
 $\text{Proj}^{\text{Fin}}_{\mathbb{K}}$ has a natural symmetric tensor product and pivotal structure which has trivial twists \cite{TVVA}*{Example 1.7.2}. 
Note that  $\text{Proj}^{\text{Fin}}_{\mathbb{K}}=\text{Vect}^{\text{Fin}}_{\mathbb{K}}$ when $\mathbb{K}$ is a field.
\end{ex}

\begin{ex}
\label{ex:category of super vector spaces}
For a ribbon Hopf algebra $A$ with the   R-matrix $R$ which is symmetric, i.e., $\Delta (R)=\Delta^{\text{op}}(R)$, the finite dimensional representation category 
$\text{Mod}^{\text{Fin}}_A$ of $A$ admits a symmetric pivotal structure. 
Note that the group algebra $\mathbb{K}[\mathbb{Z}_2]$  for a field $\mathbb{K}$
admits a ribbon Hopf algebra structure with the R-matrix 
$\frac{1}{2}(1\otimes 1 + 1\otimes g + g\otimes 1 -g\otimes g)$ and the ribbon element $1$, where $g$ is a generator of $\mathbb{Z}_2$.
Then $\text{Mod}^{\text{Fin}}_{\mathbb{K}[\mathbb{Z}_2]}$ has a symmetric tensor product and pivotal structure with trivial twists, which can be transported  to the category $\text{SVect}^{\text{Fin}}_{\mathbb{K}}$ of finite dimensional super vector spaces, under the identification of $\mathbb{Z}_2$ modules and super vector spaces. 
\end{ex}

\subsection{Tensor Network}
\label{ssec:tn} 

A \textit{\textbf{tensor network}} over a finite dimensional vector space $V$ is an oriented graph which represents a tensor labeled by the set of open edges, where each incoming (resp. outgoing) edge labels $V^*$ (resp. $V$).
For example, the diagram in Figure \ref{fig:T and F} presents a $(m,n)$ tensor $T\in  (V^*)^{\otimes\mathcal{I}}\otimes V^{\otimes \mathcal{O}}=\mathrm{Hom}(V^{\otimes\mathcal{I}}, V^{\otimes\mathcal{O}})$, where $\mathcal{I}=\{i_1,\ldots, i_m\}$ is the set of incoming edges and $\mathcal{O}=\{o_1,\ldots, o_n\}$ is the set of outgoing edges.

\begin{figure}[ht]
  \centering
  \includegraphics{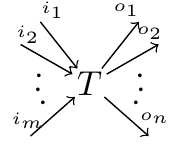}
  \caption{$(m,n)$ tensor.}
  \label{fig:T and F}
\end{figure}

Given two tensor networks $T$ and $S$, one gets a new tensor network by connecting an outgoing edge $o$ of $T$ with an incoming edge $i$ of $S$ (see Figure \ref{fig:contruction}), which represents the tensor obtained from $T\otimes S$ by contracting $V_o$ and $(V^*)_i$. 
\begin{figure}[H]
  \centering
  \includegraphics{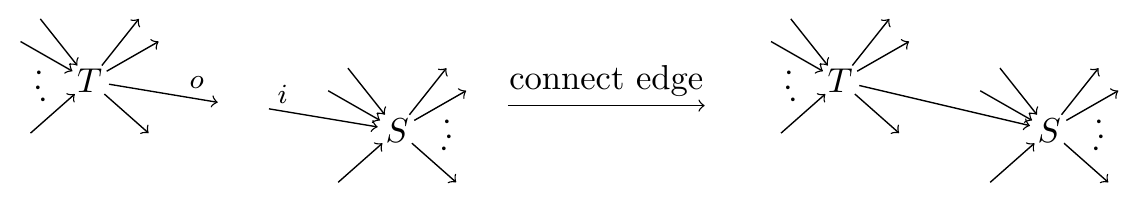}
  \caption{Contraction of tensors.}
  \label{fig:contruction}
\end{figure}

 For example, the left diagram in Figure \ref{fig:f and g} represents the composition $g\circ f$ of two maps $f, g\co V\to V$ and the right diagram represents the trace $\sum_{i}f^i_i = \text{Tr}(f)
\in \mathbb{K}$ of a map $f\co V\to V$, where $f_i^i=e^i(f(e_i))$ for a basis vector $e_i$ and its dual $e^i$. 
\begin{figure}[H]
  \centering
  \includegraphics{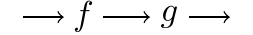}
  \hspace{10mm}
  \includegraphics{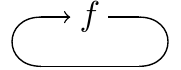}
  \caption{$g\circ f$ and $\text{Tr}(f)$.}
    \label{fig:f and g}
\end{figure} 

\noindent Here, reading from left to right,  the cap and the cup diagrams above can be seen as the coevaluation 
\begin{align*}
\text{coev}_V\co \mathbb{K} \to V \otimes V^*, \quad 1\mapsto \sum_i e_i\otimes e^i,
\end{align*}
 and the evaluation
\begin{align*}
 \widehat{\text{ev}_V}\circ V\otimes V^* \to \mathbb{K}, \quad x\otimes f \mapsto f(x),
\end{align*}
respectively, in $\text{Vect}^{\text{Fin}}_{\mathbb{K}}$. For example, we have
$\text{Tr}(f)= \widehat{\text{ev}_V}\circ (f\otimes\text{id}_V) \circ  \text{coev}_V \in \text{Hom} _{\text{Vect}^{\text{Fin}}_{\mathbb{K}}}(\mathbb{K}, \mathbb{K})$.

We can define \textit{\textbf{tensor networks in a symmetric pivotal category $\mathcal{C}$ with trivial twists}} in a similar way, and in the following sections we use them 
to describe morphisms in $\mathcal{C}$.
As a morphism in $\mathcal{C}$, a tensor network will normally be read from left to right, and the end points will be implicitly labeled (numbered) from top to bottom.
For example, for a monoid $A$ and a comonoid $C$ in $\mathcal{C}$, the multiplication $M\co A_1\otimes A_2 \to A$ and the comultiplication $\Delta \co  C \to C_1\otimes C_2$ are described as

\begin{figure}[H]
  \centering
  \includegraphics{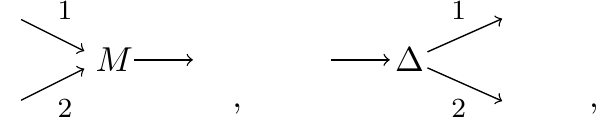}
\end{figure} 
\noindent respectively. 

\subsection{Hopf monoid}
Let $\mathcal{C}$ be a symmetric pivotal category with trivial twists.
\begin{defi}[Hopf monoid]\label{def:Hopf monoid}
A \textit{\textbf{Hopf monoid}} $(H, M, 1, \Delta, \epsilon, S)$ in  $\mathcal{C}$ is an object $H$ equipped with five morphisms
\begin{align*}
&M\co H\otimes H \to H, \quad
1\co \mathbb{I} \to H,\quad
\Delta \co H  \to H\otimes H,\quad 
\epsilon \co H \to \mathbb{I},\quad
S\co H \to H,
\end{align*}
called the \textit{\textbf{product}}, the \textit{\textbf{unit}}, the \textit{\textbf{coproduct}}, the \textit{\textbf{counit}} and the \textit{\textbf{antipode}}, respectively, which satisfies the following axioms:
\begin{figure}[H]
  \centering
  \includegraphics{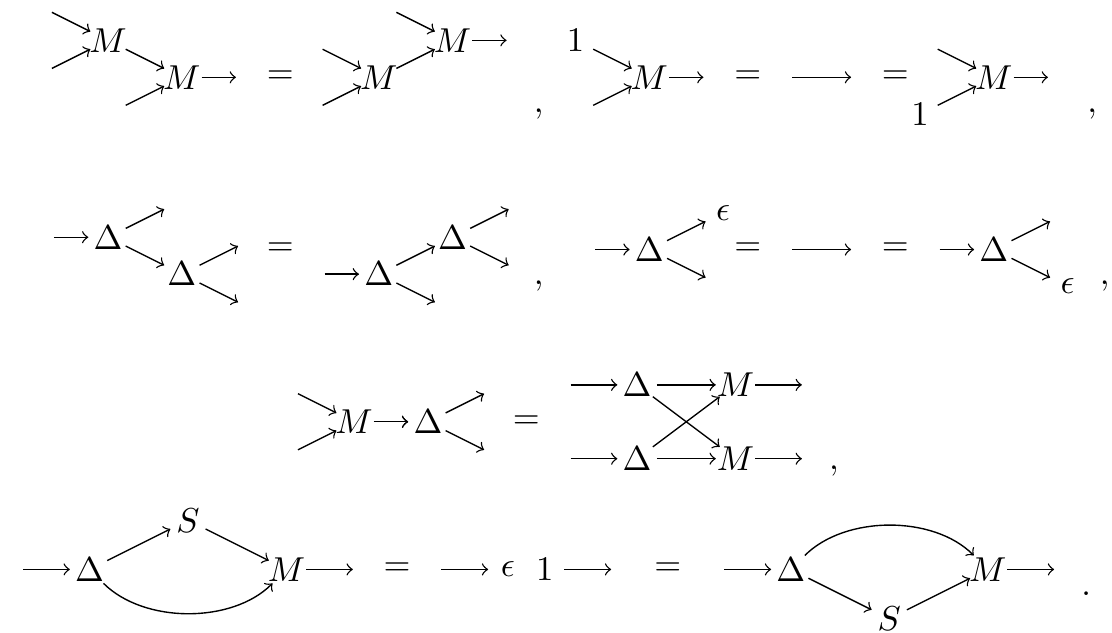}
\end{figure}
\end{defi}

Note that a Hopf monoid in $\text{Vect}^{\text{Fin}}_{\mathbb{K}}$ is a Hopf algebra over $\mathbb{K}$, and basic properties of Hopf algebras also hold for Hopf monoids as follows, see e.g. \cite{Ku2} for the proofs.
\begin{prop}\label{prop:antipode is antialgebra}
The following equality holds for any Hopf monoid in  $\mathcal{C}$:
  \begin{figure}[H]
      \centering
      \includegraphics{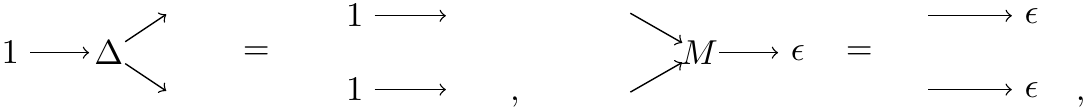}
  \end{figure}
\begin{figure}[H]
  \centering
  \includegraphics{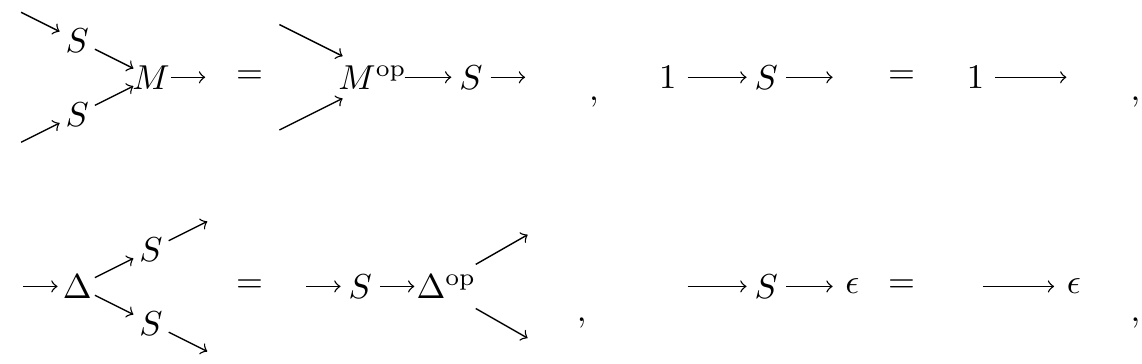}
\end{figure}
\end{prop}
where $M^{\text{op}}$ and $\Delta^{\text{op}}$ are defined as follows:
\begin{figure}[H]
  \centering
  \includegraphics{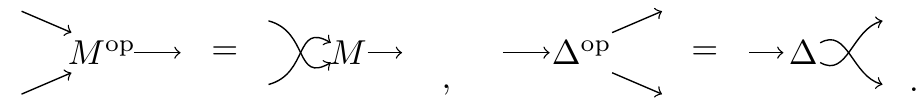}
\end{figure}

\begin{prop}\label{prop:antipode has inv}
  For a Hopf monoid in $\mathcal{C}$, there exists the inverse of the antipode $S$ which satisfies the following:
  \begin{figure}[H]
    \centering
    \includegraphics{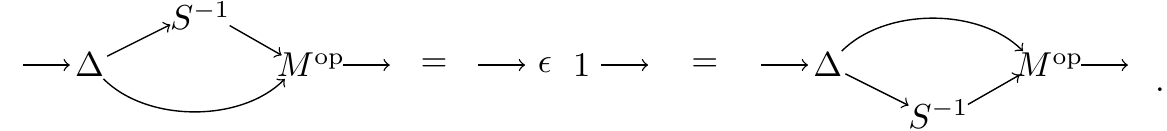}
  \end{figure}
  Since $S$ is an anti-(co)algebra morphism, $S^{-1}$ is also anti-(co)algebra morphism, i.e., we have
  \begin{figure}[H]
      \centering
      \includegraphics{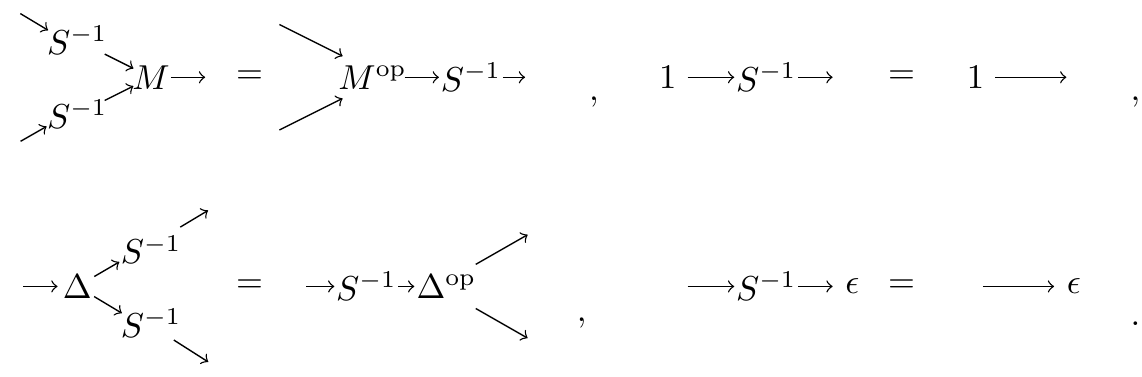}
  \end{figure}
\end{prop}

The group algebra $\mathbb{K}[G]$ for a finite group $G$ is a basic example of a finite dimensional involutory Hopf algebra.  
Here we see two more examples; 
 Example \ref{quantum Borel subalgebra} is a non-involutory Hopf algebra, and Example \ref{ex:exterior algebra } is a involutory Hopf monoid in $\text{SVect}_{\mathbb{K}}$ which is not a Hopf algebra.

\begin{ex}
  \textbf{Small quantum Borel subalgebra $u_q(\mathfrak{sl}_2^+)$}
  \label{quantum Borel subalgebra}

  Let $q \in \mathbb{C}$ be a fixed $n$-th primitive root of unity.
  The small quantum group $u_q(\mathfrak{sl}_2^+)$ associated with the Borel subalgebra $\mathfrak{sl}_2^+$ of the Lie algebra $\mathfrak{sl}_2$ is the algebra over $\mathbb{C}$ defined by  generators $E, K$ and relations
  \begin{align*}
  KE=qEK, \quad E^n=0, \quad K^n=1.
  \end{align*}
 The Hopf algebra structure on $u_q(\mathfrak{sl}_2^+)$ is defined by
  \begin{gather*}
  \Delta(E)=E\otimes 1+K\otimes E ,\quad\Delta(K)=K\otimes K,\\
  \epsilon(E)=0,\quad \epsilon(K)=1,\quad S(E)=-K^{-1}E,\quad S(K)=K^{-1}.
  \end{gather*}
  In general  we have a Hopf algebra $u_q(\mathfrak{g}^+)$
for any complex simple Lie algebra $\mathfrak{g}$.
\end{ex}

\begin{ex}
  \textbf{Exterior algebra $\Lambda(V)$}
  \label{ex:exterior algebra }

  Let $V$ be a finite dimensional vector space, and let $\Lambda(V)$ be its exterior algebra.
  For $X_i\in V\subset \Lambda(V)$, set $\text{deg}(X_1\wedge\cdots\wedge X_p)=1$ for $p$ odd and $\text{deg}(X_1\wedge\cdots\wedge X_p)=0$ for $p$ even.
  Then $\Lambda(V)$ becomes a super vector space, i.e. an object in the symmetric pivotal category $\text{SVect}_{\mathbb{K}}$ in Example \ref{ex:category of super vector spaces}.
  The Hopf monoid structure is defined by
  \begin{center}
  $\Delta(X)=X\otimes 1+1\otimes X \quad ,\epsilon(X)=0, \quad S(X)=-X$,\\
  \end{center}
  where $X\in V\subset \Lambda(V)$.
\end{ex}

\section {Invariant of integral normal o-graphs}
\label{sec:IF}
In this section, we introduce integral normal o-graphs and  construct their invariant up to certain moves. We will see in  Section \ref{sec:Branched standard spine} that this invariant  gives an invariant of closed framed $3$-manifolds with the vanishing first Betti number.

\subsection{Integral normal o-graph and integral 0-2 move, integral MP-move and H-move}
\label{sec:BP-diagrams}
\begin{defi}[Integral normal o-graph]
\label{def:bp}
An \textit{\textbf{integral normal o-graph}} is an oriented virtual link diagram with an integer weight attached on each edge, i.e., 
a finite connected $4$-valent graph $\Gamma$ immersed in $\mathbb{R}^2$, where the self-intersections are transverse double points,  with the following conditions:
\begin{spacing}{0.5}
\end{spacing}
\begin{description}
\setlength{\itemsep}{1mm}
\setlength{\parskip}{1mm} 
\item[N1] At each vertex, a sign $+$ or $-$ is indicated, which is presented by the over-under notation as in Figure \ref{fig:vertex},
\item[N2] Each edge is oriented, and  the orientations of two edges which are opposite to each other at a vertex match,
\item[W1] To each edge an integer is attached.
\end{description}
\begin{figure}[ht]
  \centering
  \includegraphics{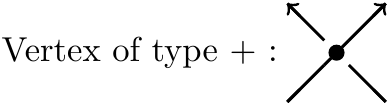}
  \hspace{20mm}
  \includegraphics{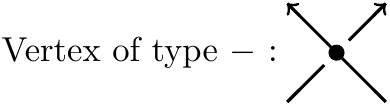}
  \caption{Type of vertices.}
  \label{fig:vertex}
\end{figure}
\end{defi}

Note that there are two kind of crossings: real crossings (i.e., a vertices of the graph)  and  virtual crossings (i.e.,  singular points of the immersion of the graph).
When we refer to a crossing, we mean a real crossing. 
We consider integral normal o-graphs modulo Reidemeister type moves in Figure \ref{fig:RM},  and  denote the set of equivalent classes by $\mathcal{IG}$.

\begin{figure}[H]
    \centering
    \includegraphics{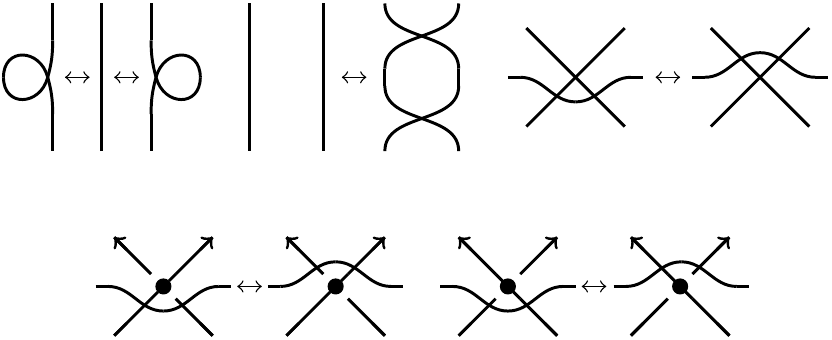}
    \caption{Reidemeister type moves.}
    \label{fig:RM}
  \end{figure}

We define the following three types of  moves on $\mathcal{IG}$ and denote by $\sim$ the generated equivalence relation.

\begin{itemize}
\setlength{\itemsep}{1mm}
\setlength{\parskip}{1mm} 
\item integral 0-2 move (Figure \ref{fig:0-2 move})
\item integral MP-move (Figure \ref{fig:MP-move})
\item H-move (Figure \ref{fig:H move})
\end{itemize}

  \begin{figure}[H]
  \begin{minipage}[b]{0.4\linewidth}
    \centering
    \includegraphics[scale=0.9]{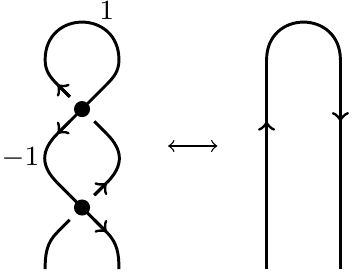}
    \caption{Integral 0-2 move.}
    \label{fig:0-2 move}
  \end{minipage}
\end{figure}
\begin{figure}[H]
    \centering
    \includegraphics[scale=0.8]{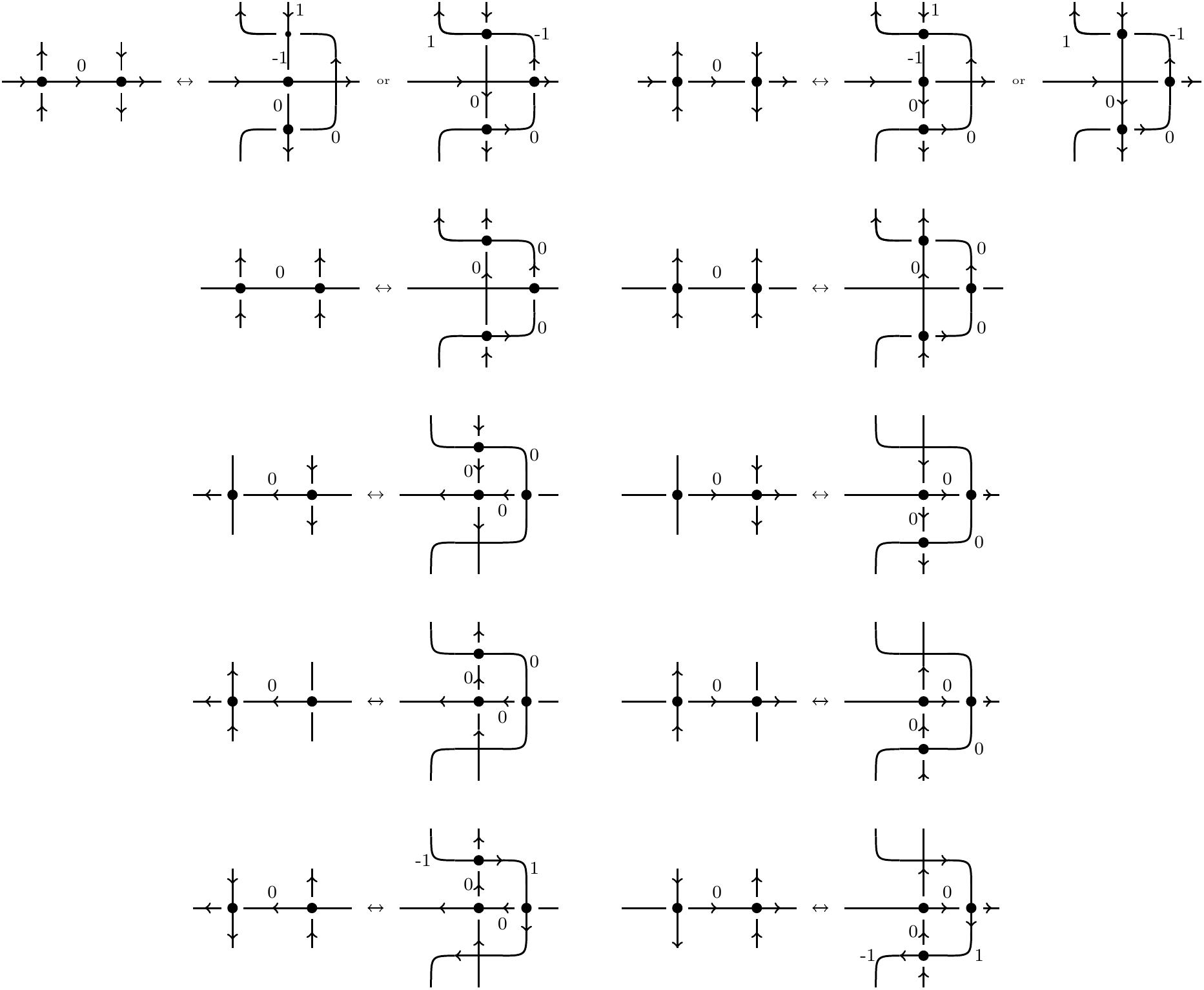}
    \caption{Integral MP-move.}
    \label{fig:MP-move}
\end{figure}

\begin{figure}[H]
    \centering
    \includegraphics{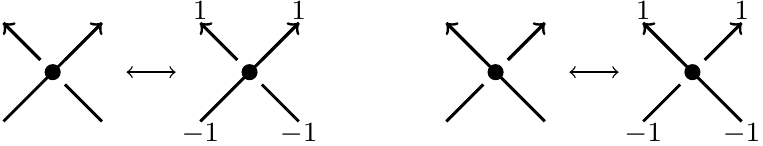}
    \caption{H-move.}
    \label{fig:H move}
\end{figure}

\noindent Here, in Figure \ref{fig:MP-move}, the orientations of the non-oriented edges are arbitrary if they match before and after the move. If there are multiple weights on an edge after the move, the weights should be added in the additive group $\mathbb{Z}$. 

\subsection{Definition of invariant}
\label{sec:Invariant}
Let $H=(H,M,1,\Delta,\epsilon,S)$  be a Hopf monoid  in a symmetric pivotal category with trivial twist.
We construct the invariant 
\begin{align*}
Z(-;H) \co \mathcal{IG} \ \to \ \text{End}(\mathbb{I})
\end{align*}
through the following two steps. First, we replace each vertex and edge with a tensor network shown in the Figure \ref{fig:Definition of the functor}. Second, we connect the legs of tensor networks according to how the integral normal o-graph was connected.
This defines a tensor network without any incoming or outcoming edge, i.e., an element $Z(\Gamma;H)\in \text{End}(\mathbb{I})$.
In the case of $\mathcal{C}=\text{Vect}_{\mathbb{K}}^{\text{Fin}}$, i.e., when $H$ is a Hopf algebra, we have a scalar $Z(\Gamma;H)\in\mathbb{K}$.

\begin{figure}[H]
  \centering
  \includegraphics{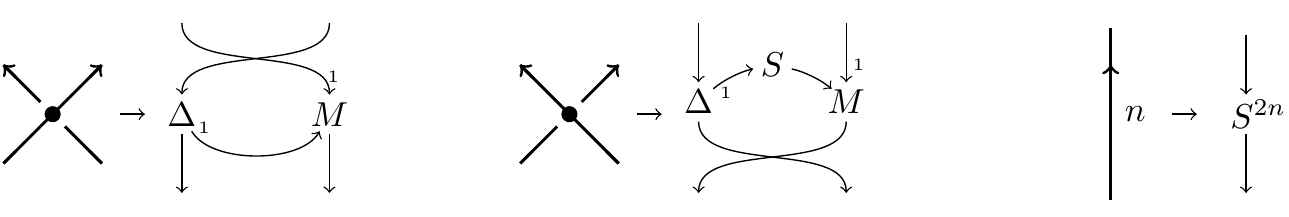}
  \caption{Definition of the invariant $Z(-;H)$.}
  \label{fig:Definition of the functor}
\end{figure}

\begin{ex}
The integral normal o-graph below will represent framed $S^3$, see Section \ref{Example of closed BP diagram}.  If $\mathcal{C}=\text{Vect}_{\mathbb{K}}^{\text{Fin}}$, then the invariant  $\text{Tr}(S^{2}\circ M^{\text{op}}\circ (\text{id}_{H}\otimes S^{4})\circ\Delta)$ turns out to be a distinguished scalar element in $\mathbb{K}$ obtained by integrals of the Hopf algebra, see
 Example \ref{ex:invariant of S3}.
\begin{figure}[H]
    \centering
    \includegraphics{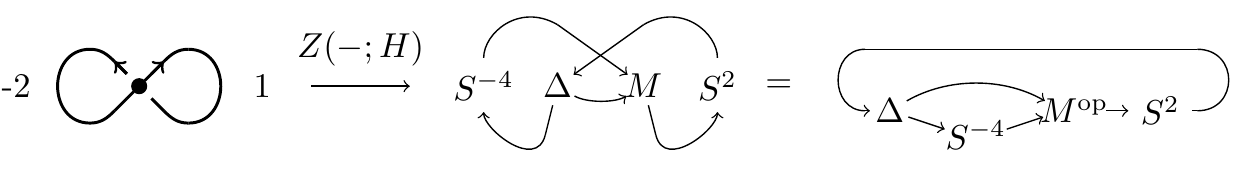}
    \label{fig: example of invariant}
\end{figure}
\end{ex}

\begin{thm}
Let $H$ be a Hopf monoid in a symmetric pivotal category $\mathcal{C}$ with trivial twists. For an integral normal o-graph $\Gamma$, $Z(\Gamma;H)$ is invariant under the moves in Figure \ref{fig:RM} to Figure \ref{fig:H move}, i.e.,  the map
\begin{align*}
    Z(-;H)\co \mathcal{IG}/\sim \ \to \ \text{End}(\mathbb{I})
\end{align*} is well-defined.
\end{thm}

\begin{proof}

\textbf{Invariance under the Reidemeister type moves (Figure \ref{fig:RM}).} The invariance under the RI type move follows from the fact that  $\mathcal{C}$ has trivial twists (cf. Definition \ref{def:trivial twists}).  That of the RII and the RIII type moves are from the  invertibility and the naturality, respectively,  of the symmetry.

\textbf{Invariance under the integral 0-2 move (Figure \ref{fig:0-2 move}).}
For the integral 0-2 move, we have
\begin{figure}[H]
    \centering
    \includegraphics{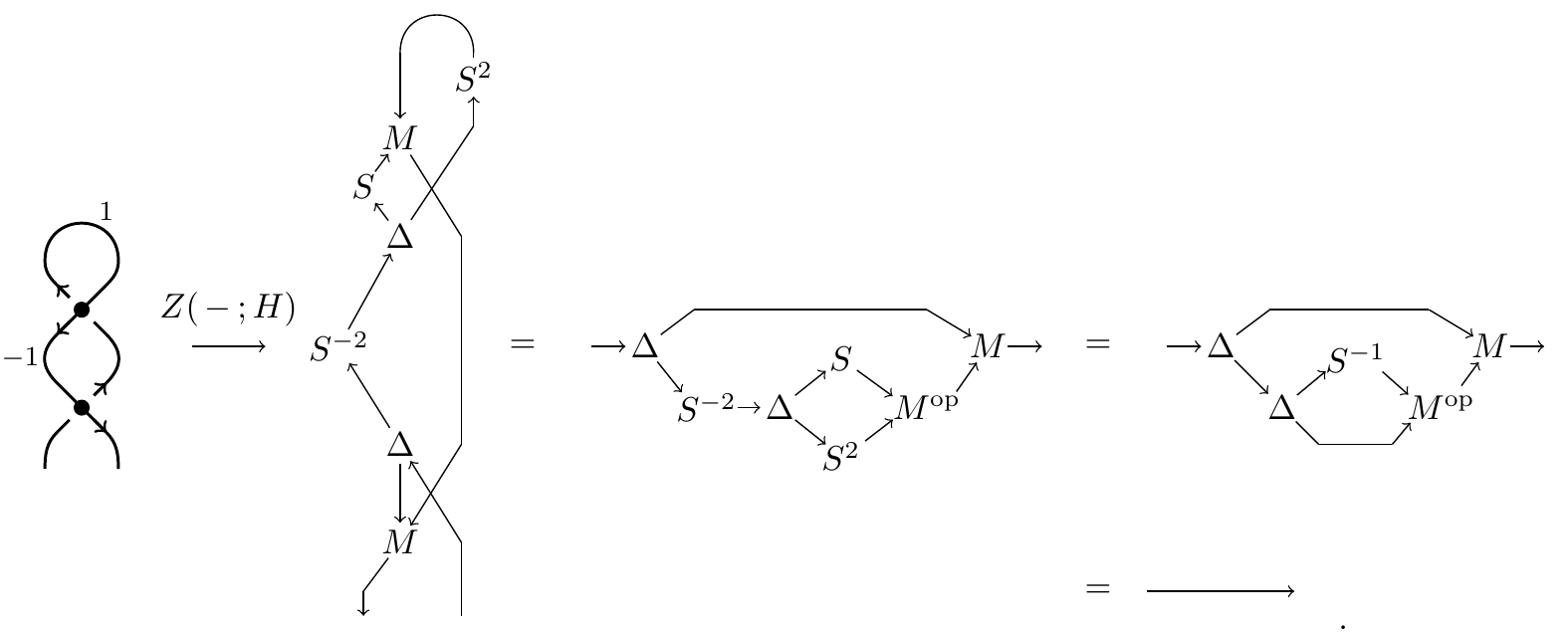}
\end{figure}
\noindent Here,  the second equality follows from the fact that $S^2$ is a coalgebra morphism, and the third equality follows from Proposition \ref{prop:antipode has inv}.

\textbf{Invariance under the MP-move (Figure \ref{fig:MP-move}).}
We divide the 16 MP-moves into two types: ones with all the edges colored by 0 (Type A) and the others with two edges colored by $\pm1$ (Type B).
The moves of Type B are in the top and the bottom row in Figure \ref{fig:MP-move},  and the moves of Type A are in the middle rows.
In \cite{S} and \cite{MST}, it was proved, in terms of the canonical element of the Heisenberg double of the Hopf algebra, that $Z(-;H)$ is invariant under the MP-moves of Type A for any finite dimensional Hopf algebra. These proofs can be easily generalized to Hopf monoids in  symmetric pivotal categories with trivial twists.
Thus we need only to show the invariance under the MP-moves of Type B.  

We prove the top left move in Figure \ref{fig:MP-move}. We can prove the other moves similarly. For the LHS, we have

\begin{figure}[H]
    \centering
    \includegraphics[scale=0.9]{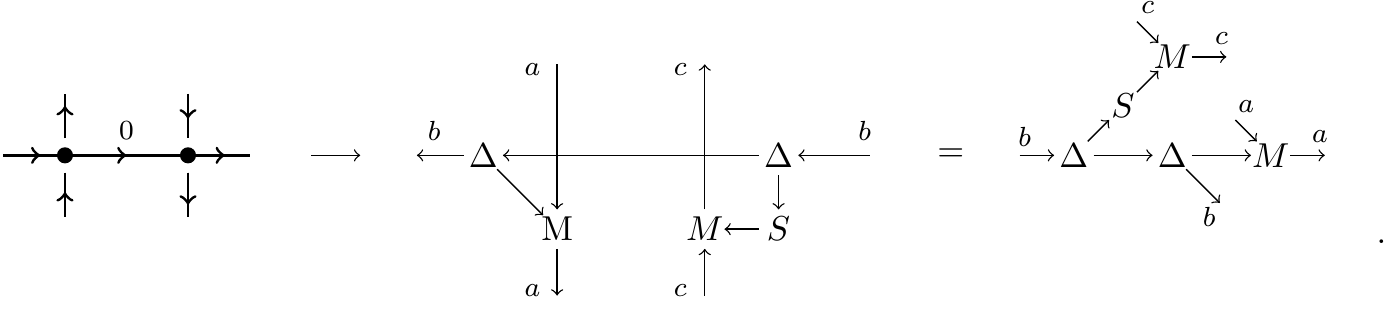}
\end{figure}
\noindent 
\noindent Here, the labels $a, b ,c$ are attached to endpoints so that we can trace them under deformations of the tensor networks.
For the RHS,  we have

\begin{figure}[H]
    \centering
    \includegraphics[scale=0.9]{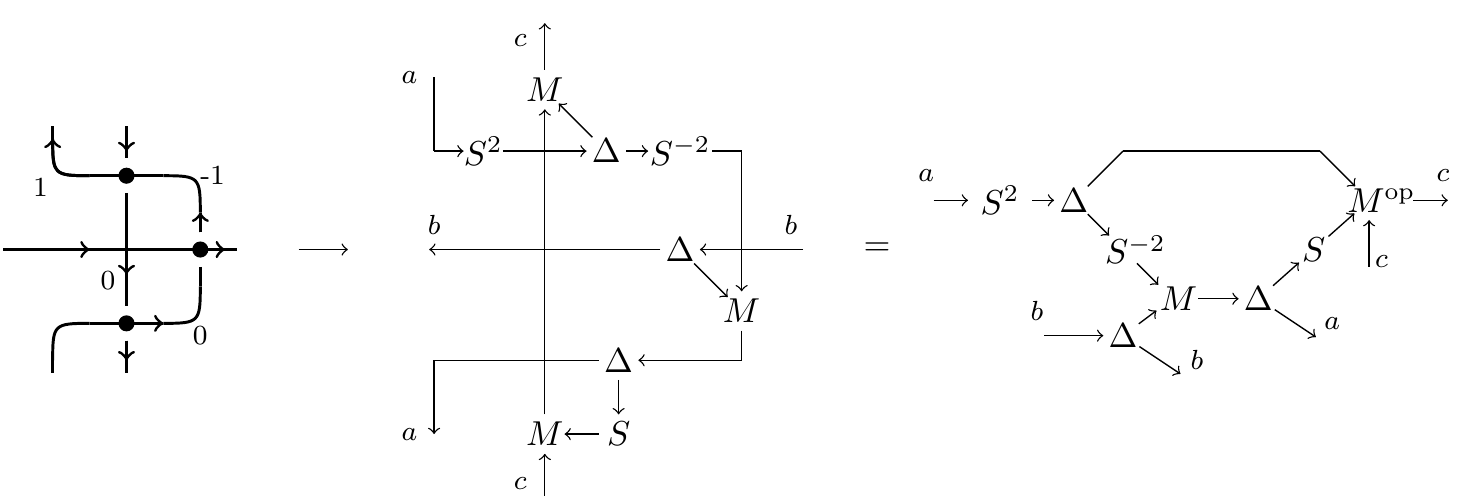}
\end{figure}
\begin{figure}[H]
    \centering
    \includegraphics[scale=0.9]{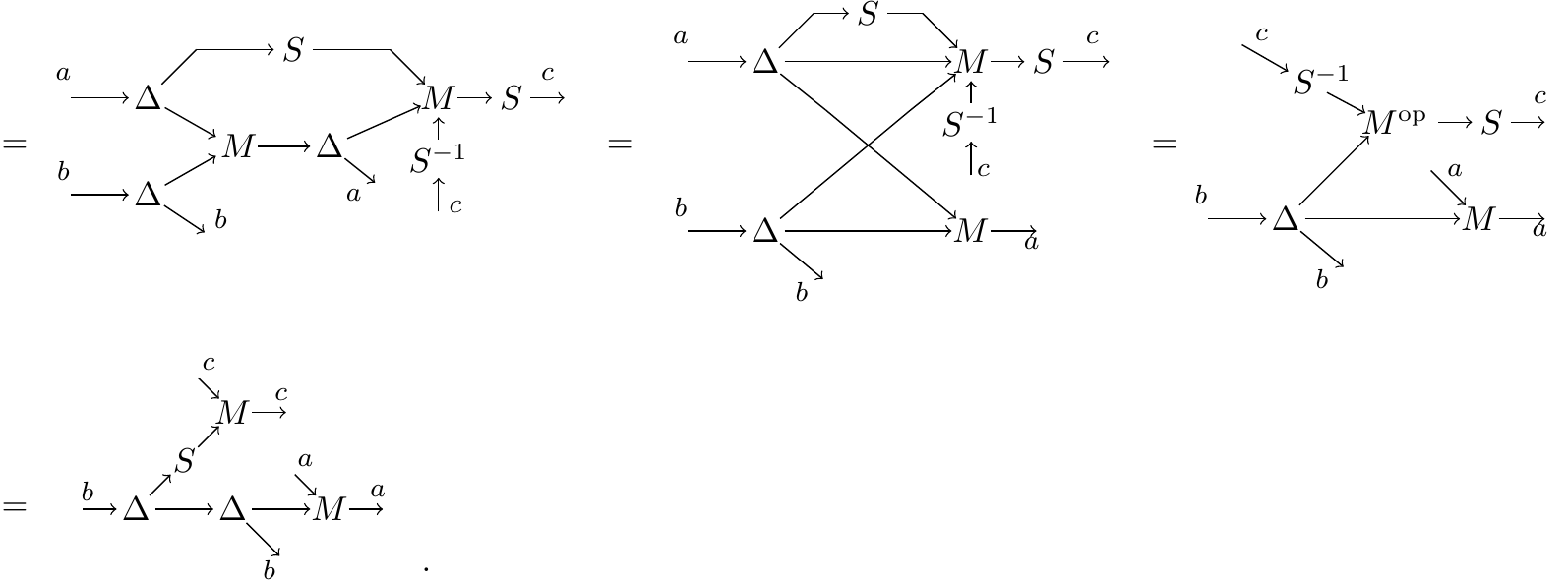}
\end{figure}

\noindent  Here the first equality follows from the isotopy invariance of tensor networks.
Other equalities follow from the axiom of Hopf monoids and the fact that the antipode $S$ is an anti-(co)algebra morphism. 

\textbf{Invariance under the  H-move (Figure \ref{fig:H move}).}
For H-move, by the fact that $S^2$ is a (co)algebra morphism, we have
\begin{figure}[H]
    \centering
    \includegraphics{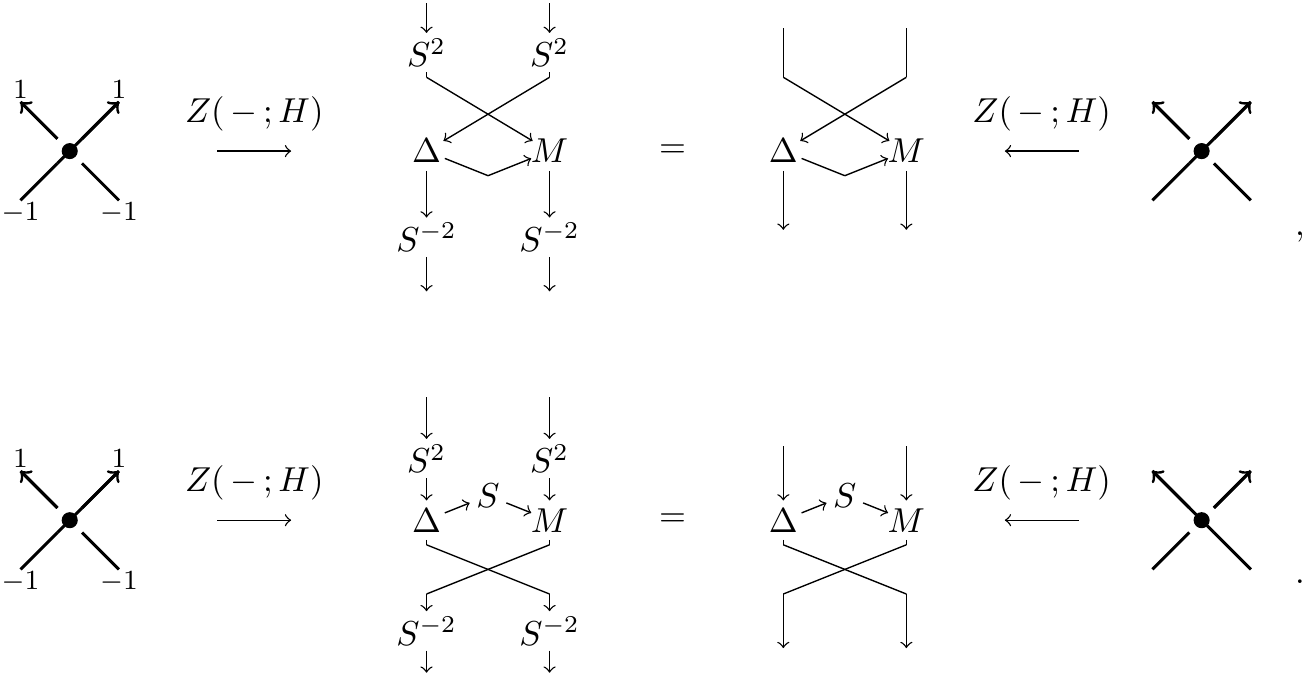}
\end{figure}

\end{proof}

\begin{rem}
The above construction can be thought as a generalization of the invariants constructed in \cite{MST} using involutory Hopf algebras; when $S^2=\text{id}_H$, the invariant does not depend on the integer weights and the construction  coincides with the one in \cite{MST}.
In \cite{MST}, we constructed the invariant as a functor from the category of ``o-tangles", and we can similarly extend the above construction to a functor from the category of ``integral normal o-tangles".
\end{rem}

\section{Framed integral normal o-graph and closed framed $3$-manifold}
\label{sec:Branched standard spine}
In this section, we introduce a subset  $\mathcal{FIG}\subset \mathcal{IG}$ consisting of \textit{\textbf{framed integral normal o-graphs}} and show that elements of $\mathcal{FIG}$ represent equivalent classes of closed framed 3-manifolds, i.e.,  we construct a surjective map 
\begin{align*}
    \Phi_{\text{fram}}\co \mathcal{FIG} \to \mathcal{M}_{\text{fram}},
\end{align*} 
 where  $\mathcal{M}_{fram}$ is the set of equivalent classes of closed framed $3$-manifolds (see Section \ref{subsec: framing} for the details).  
Then we prove the following.
\begin{prop}\label{prop: framed 3-mfd BP-diagram}
Let $\mathcal{M}^0_{fram}\subset \mathcal{M}_{fram}$ be the subset consisting of closed framed $3$-manifolds with the vanishing first Betti number, and $\mathcal{FIG}^0\subset \mathcal{FIG}$  the inverse image of $\mathcal{M}^0_{fram}$ by $\Phi_{\text{fram}}$.
The restriction of  $\Phi_{\text{fram}}$ induces a bijection
\begin{align*}
    \Phi^0_{\text{fram}}\co \mathcal{FIG}^0/\sim \ \to \ \mathcal{M}^0_{\text{fram}},
\end{align*} 
where $\sim$ is the restriction of the equivalent relation  on $\mathcal{IG}$ defined in Section \ref{def:bp}.
\end{prop}
As a result, the restriction of the invariant $Z(-,H)$ to $\mathcal{FIG}^0$ gives an invariant of closed framed $3$-manifolds with the vanishing first Betti number.

In Sections \ref{subsec:Branched spine}--\ref{closed normal o-graph} we give an outline of standard arguments in \cite{BP} to represent closed framed $3$-manifolds by graphs.
In Section \ref{subsec:Branched spine} and \ref{subsec:combing of 3-manifold} we recall  branched spines representing combed $3$-manifolds. Section \ref{subsec: framing} is for extending these combings to framings. Then in Section \ref{closed normal o-graph} we recall normal o-graphs  to represent branched spines combinatorially.
 In Section \ref{subsec:BP diagram} we define framed integral normal o-graphs and prove Proposition \ref{prop: framed 3-mfd BP-diagram}, where we modify the arguments in \cite{BP} by lifting $\mathbb{Z}/2\mathbb{Z}$ weights on edges on  normal o-graphs to integer weights.
 We give examples of framed integral normal o-graphs in the last section. In what follows all manifolds and polyhedrons are assumed to be connected.

\subsection{Branched spine and associated vector field}
\label{subsec:Branched spine}

For an introduction to standard and branched spines, see for e.g. \cites{BP,Mat}\footnote{In \cite{Mat}, standard spine is called special spine.}.
A 2-dimensional compact polyhedron $P$ is called \textit{\textbf{simple}} if the neighborhood of each point $x\in P$ is homeomorphic to one of the pictures in Figure \ref{fig:nbh of simple spine}, where from the left in the picture the point $x$ is called a \textit{\textbf{nonsingular point}}, a \textit{\textbf{triple point}}, and a \textit{\textbf{true vertex}}, respectively.

\begin{figure}[H]
    \centering
    \includegraphics{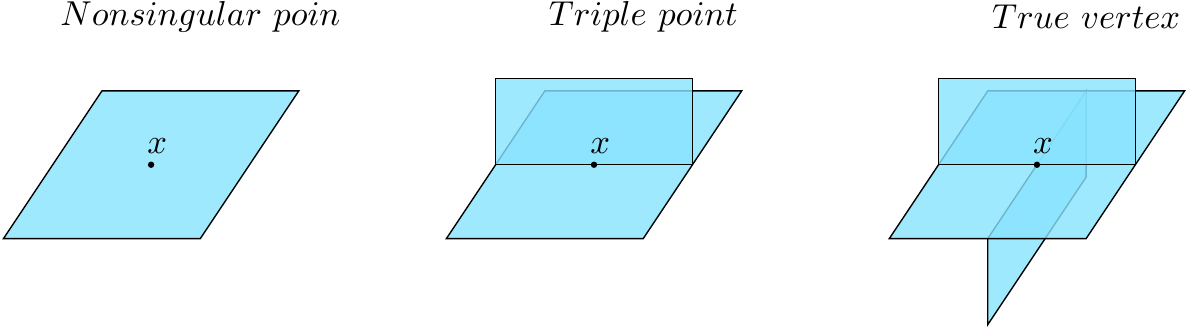}
    \caption{Nonsingular point, triple point, and true vertex.}
    \label{fig:nbh of simple spine}
\end{figure}

\noindent For a simple polyhedron $P$, set
\begin{align*}
  V(P)&=\{x\in P\mid \text{$x$ is a true vertex}\},\\
  S(P)&=\{x\in P\mid \text{$x$ is a true vertex or a triple point}\},\\
  D(P)&=\{x\in P\mid \text{$x$ is a nonsingular point}\}.
\end{align*}

A simple polyhedron $P$ is called \textit{\textbf{standard}} if connected components of $S(P)\backslash V(P)$ and $D(P)$ are 1-cells and 2-cells, respectively.
Let $M$ be a 3-manifold with non-empty boundary.
A standard polyhedron $P$ embedded in $\text{Int}M$ is called a \textit{\textbf{standard spine}} of $M$ if $M$ collapses to $P$.
It is known that every compact 3-manifold with non-empty boundary admits a standard spine \cite{Mat}*{Theorem 1.1.13}.
A standard spine $P$ of $M$ determines the homeomorphism class of $M$, i.e., if $P^{\prime}$ is a standard spine of $M^{\prime}$ which is homeomorphic to $P$, then $M^{\prime}$ is homeomorphic to $M$ \cite{Mat}*{Theorem 1.1.17}. On the other hand, not all standard polyhedrons are spines of 3-manifolds,  and the condition for standard polyhedrons to be spines is described as a combinatorial condition for 2-cells attached to $S(P)$ \cite{BP2}*{Theorem 1.5}. 

For oriented 3-manifolds we can describe rather simply the conditions for standard polyhedrons to be spines.
An \textit{\textbf{oriented standard spine}} is a standard spine of an oriented 3-manifold \cite{BP}*{Proposition 2.1.2}, which can be characterized combinatorially as follows.
For each 1-cell $e$ in $S(P)\backslash V(P)$ we specify a pair $(d(e),o(e))$ of a direction $d(e)$ of $e$ and cyclic ordering $o(e)$ among the (local) three disks attached to $e$, see Figure \ref{fig:orientation of polyhedron}, up to simultaneous reversing $(d(e),o(e))\mapsto (-d(e),-o(e))$. 
Then the standard polyhedron is called \textit{\textbf{oriented}} if the orientation around  each vertex is compatible as in the Figure \ref{fig:oriented standard polyhedron}, i.e., if all of the pairs $(d,o)$ for the four edges attached to the vertex are of the same right- or left-handed type (the type depends on an embedding in $\mathbb{R}^3$ of a neighborhood of the vertex).
Every standard spine of oriented 3-manifold with non-empty boundary is canonically oriented, and an oriented standard spine determines the oriented 3-manifold up to orientation preserving homeomorphism.

\begin{figure}[H]
  \begin{minipage}[b]{0.45\linewidth}
    \centering
    \includegraphics[keepaspectratio, scale=1]{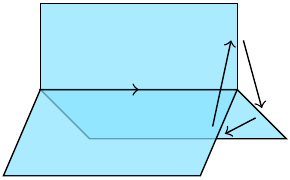}
    \caption{Orientation of polyhedron.}
    \label{fig:orientation of polyhedron}
  \end{minipage}
  \begin{minipage}[b]{0.45\linewidth}
    \centering
    \includegraphics[keepaspectratio, scale=1]{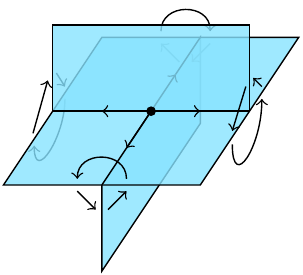}
    \caption{Orientation near vertex.}
    \label{fig:oriented standard polyhedron}
  \end{minipage}
\end{figure}

An \textit{\textbf{oriented branching}} on an oriented standard polyhedron $P$ is an orientation on connected components ($2$-cells) of $D(P)$ such that on each $1$-cell, the orientations induced from the $2$-cells attached to it are not compatible, i.e., there are locally three $2$-cells which are attached to a $1$-cell $e$ and one of the three induced orientations on $e$ is opposite to the other two (cf. \cite{BP}*{Corollary 3.1.7}). 
We can visualize a branching structure on $P$ as a smoothing of $P$  as shown in Figure \ref{fig:branching}, where the ``branching" starts from the region which induces inverse orientation on the $1$-cell relative to the others.
Here, each 1-cell has a canonical orientation as the two compatible orientations induced by the $2$-cells attached to it. 
Up to orientation preserving homeomorphism, there exist two possibilities for the branching structure near the 0-cell, the \textit{\textbf{type $+$}} and the \textit{\textbf{type $-$}},  which are shown in Figure \ref{fig:branched vertex}. 

\begin{figure}[H]
  \begin{minipage}[b]{0.3\linewidth}
    \centering
    \includegraphics[keepaspectratio, scale=1]{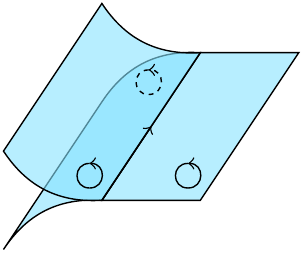}
    \caption{Branching.}
    \label{fig:branching}
  \end{minipage}
  \begin{minipage}[b]{0.6\linewidth}
    \centering
    \includegraphics[keepaspectratio, scale=1]{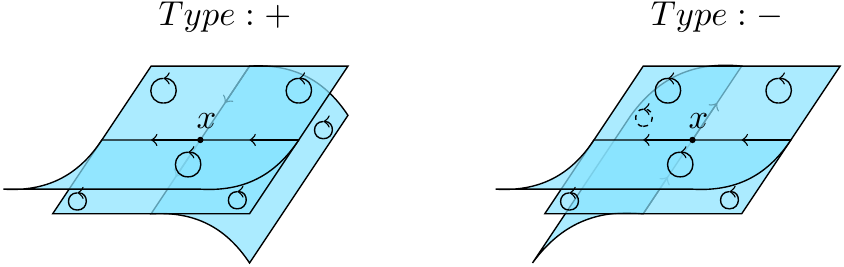}
    \caption{Local branching near  true vertex.}\label{fig:branched vertex}
  \end{minipage}
\end{figure}

\begin{rem}
If one prefers to work with ideal triangulations of oriented $3$-manifolds, one can take the dual polyhedron of oriented standard spine (Figure \ref{spine_tetrahedra}). This correspondence is one-to-one and every definition and argument in the following subsections can also be translated in terms of ideal triangulations.
For example, in terms of this dual perspective, a branching structure is a choice of edge orientation such that the edge orientation is not cyclic at every face of ideal triangulation.  
\end{rem}

\begin{figure}[H]
    \centering
    \includegraphics{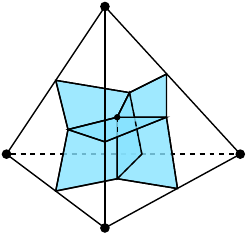}
    \caption{Ideal tetrahedron corresponding to true vertex.}
    \label{spine_tetrahedra}
\end{figure}

By abusing the terminology we call an oriented standard polyhedron endowed with an oriented branching a \textit{\textbf{branched polyhedron}}.
Let $P$ be a branched polyhedron and $M(P)$ the 3-manifold which is obtained from $P$ by thickening.
Then $P$ defines a unique non-vanishing vector field $v(P)$ on $M(P)$, perpendicular to $P$, as depicted in Figure \ref{fig:combin of branching}.
\begin{figure}[H]
 \begin{minipage}{0.48\columnwidth}
    \centering
    \includegraphics[scale=0.5]{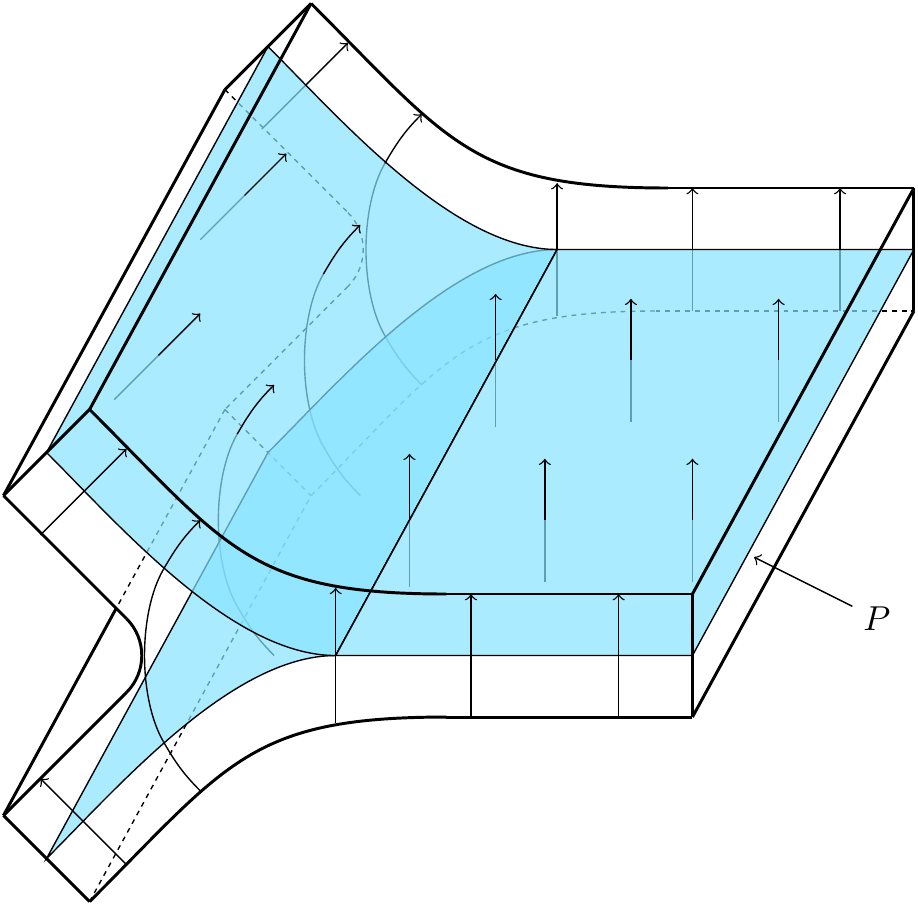}
    \caption{Vector field on $M(P)$.}
    \label{fig:combin of branching}
  \end{minipage}
    \begin{minipage}{0.5\columnwidth}
    \centering
    \includegraphics[scale=0.7]{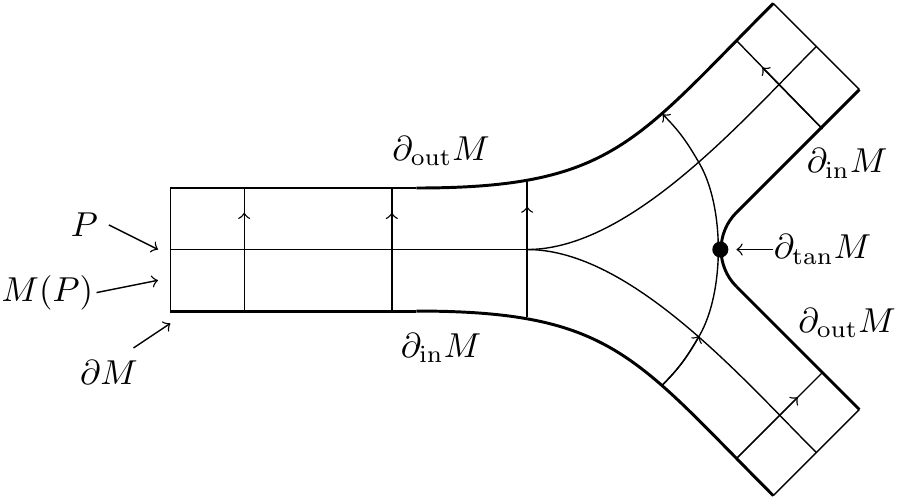}
    \caption{Stratification of boundary of $M(P)$.}
    \label{fig:black white decomposition}
  \end{minipage}
\end{figure}
\subsection{Closed combed 3-manifold}
\label{subsec:combing of 3-manifold}
Let $M$ be a
closed oriented 3-manifold.
A \textit{\textbf{combing}} $v$ of $M$ is a non-vanishing vector filed over $M$, which always exists because the Euler number of $M$ is $0$. A combed 3-manifold is a pair $(M,v)$ of a closed oriented 3-manifold $M$ and its combing $v$.
Two combed 3-manifolds $(M,v)$ and $(M^{\prime},v^{\prime})$ are \textit{\textbf{equivalent}} if there exists an orientation preserving diffeomorphism $h\co M\to M^{\prime}$ such that $h_*v$ is homotopic through combings to $v^{\prime}$. We denote by $\mathcal{M}_{\text{comb}}$ the set of equivalent classes of closed combed $3$-manifolds.

Note that for a branched polyhedron $P$, the vector field $v(P)$ of $M(P)$ induces a stratification  $\partial M=\partial_{\text{out}} M\cup \partial_{\text{tan}} M\cup\partial_{\text{in}} M$ on the boundary of $M(P)$, where 
\begin{align*}
    \partial_{\text{out}} M &= \{x\in\partial M\,|\,\text{the field $v$ at $x$ points outward}\},\\
    \partial_{\text{tan}} M &= \{x\in\partial M\,|\,\text{the field $v$ at $x$ is tangent to }\partial M\},\\
    \partial_{\text{in}} M &= \{x\in\partial M\,|\,\text{the field $v$ at $x$ points inward}\},
\end{align*}
as shown in Figure \ref{fig:black white decomposition}. A \textit{\textbf{closed branched polyhedron}} $P$ is a branched polyhedron such that $\partial M(P)=S^2_{\text{triv}}$, where $S^2_{\text{triv}}$ is the 2 dimensional sphere $S^2=\{(x,y,z)\in\mathbb{R}^3\,|\,x^2+y^2+z^2=1\}$ with the stratification given by $\partial_{\text{out}} S^2=\{z < 0\}$, $\partial_{\text{tan}} S^2=\{z = 0\}$, $\partial_{\text{in}} S^2=\{z > 0\}$.
We denote by $\mathcal{P}$ the set of homeomorphism classes of closed branched polyhedrons. 
For a closed branched polyhedron $P$, the associated vector field $v(P)$ on $M(P)$ extends to the combing $\widehat v(P)$ of closure $\widehat M(P)$ of $M(P)$ by capping the boundary $S^2_{\text{triv}}$ with a ball with trivial flow $(B^3,\frac{\partial}{\partial z})$.
Then we have a map
\begin{align*}
  \Phi_{\text{comb}}\co  \mathcal{P}\to \mathcal{M}_{\text{comb}},
\end{align*}
which is surjective \cite{BP}*{Proposition 5.2.3}.
Furthermore, the above surjection reduces to a bijection modulo branched version of Matveev-Piergallini moves on branched polyhedrons, which we will recall in Section \ref{closed normal o-graph} using graphical terminology.

\subsection{
Closed framed 3-manifold}
\label{subsec: framing} 
In what follows we sometimes use metrics of $3$-manifolds for convenience, while the results do not depend on them. Let $M$ be a closed oriented 3-manifold. 
A \textit{\textbf{framing}} $(v_1, v_2, v_3)$ of $M$ is a trivialization of the tangent bundle $TM$ such that $v_i\perp v_j$ for $i,j=1,2,3$ and $i\neq j$.
We assume that the orientation of $M$ induced from the framing matches with the existing one.
Two framed 3-manifolds are \textit{\textbf{equivalent}} if there exists a diffeomorphism $h\co M\to M^{\prime}$ such that $(h_*v_1, h_*v_2, h_*v_3)$ is homotopic through framings to $(v_1^{\prime},v_2^{\prime},v_3^{\prime})$. We denote by $\mathcal{M}_{\text{fram}}$  the
 set of equivalent classes of oriented closed framed $3$-manifolds.
Notice that we can obtain the third vector $v_3$ of a framing from $v_1$, $v_2$ and the orientation of $M$.
Thus, we represent a framed 3-manifold as $(M, v_1, v_2)$ specifying only the first two vectors $v_1, v_2$.

Let $(M,v_1)$ be a closed oriented combed 3-manifold.
Note that $v_1$ is a non-vanishing vector field, 
thus $TM$ canonically splits into the rank 1 vector bundle $\mathbb{R}v_1$ and the rank 2 vector bundle $(\mathbb{R}v_1)^{\perp}$ orthogonal to $\mathbb{R}v_1$.
The Euler class $\mathcal{E}\in H^2(M;\mathbb{Z})$ of $(\mathbb{R}v_1)^{\perp}$ is the obstruction to the existence of a framing of $M$ which extends $v_1$. 

Let $P$ a closed branched spine representing a closed combed 3-manifold $(\widehat M(P), \widehat v_1(P))$ as in the previous section. Observe that $\widehat v_1$ is perpendicular to $P$, thus $(\mathbb{R}\widehat v_1)^{\perp}$ is the ``tangent bundle" of $P$, where tangency makes sense because $P$ is branched. Then the Euler class $\mathcal{E}$ of $(\mathbb{R}\widehat v_1)^{\perp}$ is represented by the 2-cochain
\begin{align*}
    c_P = \sum_i (1-n_i/2)\hat{\Delta}_i \in C^2(P;\mathbb{Z})
\end{align*}
defined in \cite{BP}*{Proposition 7.1}, where $\hat{\Delta}_i\in C^2(P;\mathbb{Z})$ is the dual of a connected component ${\Delta}_i$ (2-cells) of $D(P)$,
and $n_i$ is a number of solid dots as in Figure \ref{fig:c_p def} on the boundary of ${\Delta}_i$.

\begin{figure}[H]
    \centering
    \includegraphics{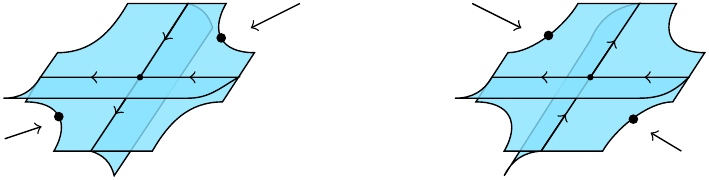}
    \caption{Solid dots on boundary of $\Delta_i$.}
    \label{fig:c_p def}
\end{figure}
\noindent Roughly speaking, $c_P$ is constructed as the obstruction to extending a certain vector field $v$ on $S(P)$ to $P$, where  $v|_{V(P)}$ is defined as in Figure \ref{sencond_combing_around_vertices} and  $v|_{S(P)\setminus V(P)}$ is defined keeping not tangent to $S(P)$. Note that the solid dots on ${\Delta}_i$ depicted in Figure \ref{fig:c_p def} correspond to the points where $v$ is tangent to $\partial {\Delta}_i$, which is used to count the rotation number $-n_i/2$ of $v$ on $\partial {\Delta}_i$ relative to the tangent vector of $\partial {\Delta}_i$, and thus the vector field $v$ needs singularities of total index $1-n_i/2$ at the interior of ${\Delta}_i$. 

\begin{figure}[H]
    \centering
    \includegraphics{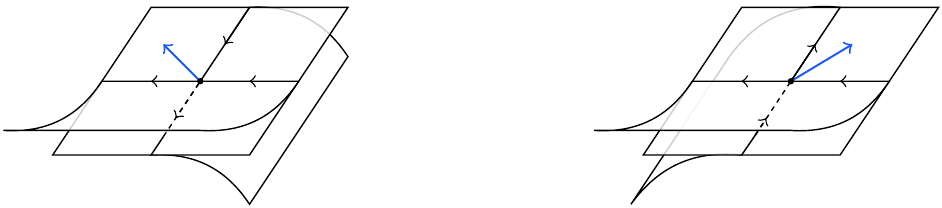}
    \caption{Local framing near vertex of type + (left) and of type - (right). }
    \label{sencond_combing_around_vertices}
\end{figure}

Let $(M, v_1, v_2)$ be a closed framed $3$-manifold and $P$ a closed branched spine representing $(M, v_1)$. We can represent the second vector $v_2$ by integer weights on $1$-cells of $P$, namely, by a 1-cochain $x\in C^1(P;\mathbb{Z})$ satisfying $\delta x= -c_P$ as follows (\cite{BS}*{Proposition 7.2.4}).\footnote{In \cite{BP}, the condition is $\delta x=c_P$ because the rotation number is counted $+1$ for clockwise rotation while we are counting oppositely. }
Using homotopy we assume that $v_2|_{V(P)}$ is as in Figure \ref{sencond_combing_around_vertices}.
Recall that each 1-cell $e$ has a canonical orientation coming from the branching of $P$.
From the starting point of $e$, let $x(e)$ be the rotation number of $v_2|_{e}$ relative to the tangent vector of $e$, where the counter clockwise $2\pi$ rotation is counted as $1$.
Then we get a 1-cochain $x\co C_1(P)\to\mathbb{Z}$, which satisfies $\delta x= -c_P$.

Conversely, given a 1-cochain $x\in C^1(P;\mathbb{Z})$ satisfying $\delta x= -c_P$, we can define the second vector $\widehat v_2(P,x)$ on $(\widehat M(P), \widehat v_1(P))$ as follows.
We first define the second vector $v_2(P,x)$ on $M(P)$ near the 0-cells as in Figure \ref{sencond_combing_around_vertices}, then we extend it over $S(P)$ so that the 1-cochain defined by rotation number matches with $x$.
Then, we extend it over the $D(P)$, where we can do this because of the boundary condition $\delta x= -c_P$ (\cite{BP}*{Proposition 7.2.1}), and since $\pi_2(S^1)=0$, the extension over the 2-cells are unique up to homotopy.
Since $\pi_2(S^1)=\pi_3(S^1)=0$, the second vector $v_2(P,x)$ defined on $M(P)$ extends uniquely to $\widehat{M}(P)$.  Thus we get the unique closed framed 3-manifold $(\widehat{M}(P), \widehat v_1(P), \widehat v_2(P,x))$ up to equivalence.  

Let $\mathcal{IP}$ be the set of pairs $(P, x)$ of homeomorphism classes of closed branched polyhedrons $P$ and 1-cochain $x\in C^1(P;\mathbb{Z})$ such that $\delta x=-c_P$. The above construction defines a surjective map 
\begin{align*}
    \Phi_{\text{fram}}\co\mathcal{IP}\to\mathcal{M}_{\text{fram}}.
\end{align*}
 
\subsection{Closed normal o-graph} \label{closed normal o-graph}
We recall from \cite{BP} closed normal o-graphs, which represent branched polyhedrons in a combinatorial manner.

For a branched polyhedron $P$, we replace each true vertex with a real crossing in $\mathbb{R}^2$; 
 we replace each vertex of type $+$ or $-$ to a crossing of type $+$ or $-$, respectively. Then we connect them according to how the true vertices are connected. Here, there could be virtual crossings depending on the combinatorics of true vertices,  and the result is an oriented virtual link diagram, which is called a \textit{\textbf{normal o-graph}}. Note that this construction is unique up to planer isotopy and the Reidemeister type moves (Figure \ref{fig:RM}).
\begin{figure}[H]
    \centering
    \includegraphics{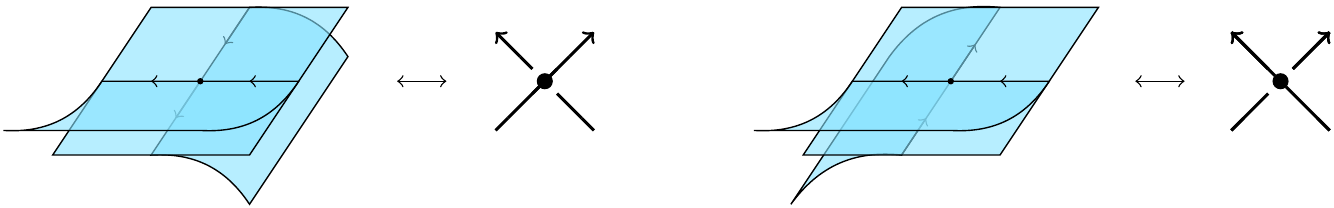}
    \caption{Correspondence between branched polyhedron and virtual link diagram near vertex of Type + (left) and type - (right).}
\end{figure}
\noindent Conversely, we can construct in a natural way from an oriented virtual link diagram a homeomorphism class of a branched polyhedron.
As a result, up to Reidemeister type moves, virtual link diagrams correspond one to one with homeomorphism classes of branched polyhedrons.

\begin{defi}[Closed normal o-graph \cite{BP}]
\label{def:closed normal o-graph}
A \textit{\textbf{closed normal o-graph}} is  
an oriented virtual link diagram  such that the associated branched spine is closed.
\end{defi}

\begin{rem}
The closed normal o-graph can also be defined combinatorially without referring the associated branched spine as in \cite{BP}*{Page 6}, where the conditions \textbf{C1}, \textbf{C2} and \textbf{C3}  for normal o-graphs  are equivalent to the closedness condition.
\end{rem}

Let $\mathcal{P}$ be the set of homeomorphism classes of closed branched spines, which will be identified to the set of closed normal o-graphs up to Reidemeister type moves.
We define the \textit{\textbf{branched 0-2 move}} and the \textit{\textbf{branched MP-move}} as in Figure \ref{fig:02} and Figure \ref{fig:MP}, respectively. 
These moves preserve the  closedness condition, and define an equivalence relation $\sim_{\text{comb}} $ on $\mathcal{P}$.

\begin{figure}[H]
    \centering
    \includegraphics[scale=0.8]{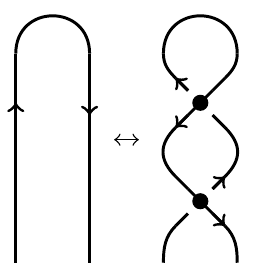}
    \caption{Branched 0-2 move.}
    \label{fig:02}
\end{figure}

\begin{figure}[H]
    \centering
    \includegraphics[scale=0.7]{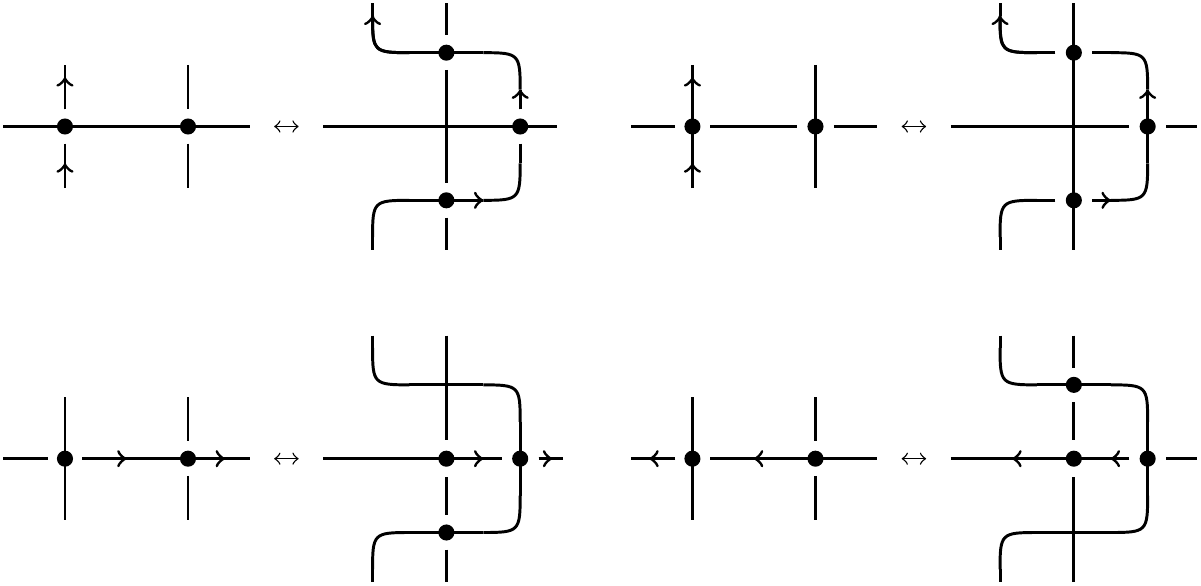}
    \caption{Branched MP-move. Orientation of each non-oriented edge is arbitrary if it matches before and after the move.}
    \label{fig:MP}
\end{figure}

\begin{prop}\cite{BP}*{Theorem 1.4.1}\label{prop:combed mfd and branched spine}
The surjective map $\Phi_{\text{comb}}\co \mathcal{P}\to \mathcal{M}_{\text{comb}}$ given in Section \ref{subsec:combing of 3-manifold}
induces a  well-defined bijection
\begin{align*}
\Phi_{\text{comb}}\co \bigslant{\mathcal{P}}{\sim_{\text{comb}}}\to \mathcal{M}_{\text{comb}}.
\end{align*}
\end{prop}

\subsection{Framed integral normal o-graph and proof of Proposition \ref{prop: framed 3-mfd BP-diagram}}
\label{subsec:BP diagram}
Note that the correspondence of normal o-graphs and branched polyhedrons implies the correspondence of integral normal o-graphs (Definition \ref{def:bp}) $\Gamma$ and pairs $(P(\Gamma),x(\Gamma))$ of  a branched polyhedron $P(\Gamma)$  and  its 1-cochain $x(\Gamma)\in C^1(P(\Gamma);\mathbb{Z})$. 
Recall from Section \ref{subsec: framing} the canonical $2$-cochain $c_P$ of a closed branched spine $P$.
\begin{defi}[Framed integral normal o-graph]
\label{def:closed-bp}
A \textit{\textbf{framed integral normal o-graph}} is an integral normal o-graph $\Gamma$  such that the associated branched spine $P(\Gamma)$ is closed and $\delta x(\Gamma)=-c_{P(\Gamma)}$.
\end{defi}

Let  $\mathcal{FIG}$ be the set of framed integral normal o-graphs up to the Reidemeister type moves. The  correspondence $\Gamma \mapsto (P(\Gamma),x(\Gamma))$ gives an identification of $\mathcal{FIG}$ and $\mathcal{IP}$.
 Note that the integral 0-2 move, the integral MP-move and the H-move preserve the conditions of framed integral normal o-graphs, and thus the equivalence relation $\sim$ is well-defined on $\mathcal{FIG}$, thus on $\mathcal{IP}$.
 Let $\mathcal{M}_{\text{fram}}^{0}$ be the subset of $\mathcal{M}_{\text{fram}}$ consisting of closed framed 3-manifolds with the first Betti number $b_1=0$, and $\mathcal{IP}^0$  the subset of $\mathcal{IP}$ consisting of pairs $(P,x)\in \mathcal{IP}$ such that $\widehat M(P)\in \mathcal{M}_{\text{fram}}^{0}$.
The following proposition can be thought as integer lift of  \cite{BP}*{Section 7.3} restricting manifolds to $b_1=0$. 

\begin{prop}
\label{framed branched spine is one to one}
The restriction of the surjective map  $\Phi_{\text{fram}}\co \mathcal{IP}\to\mathcal{M}_{\text{fram}}$ to $\mathcal{IP}^0$ defines a well-defined bijection
 \begin{align*}
    \Phi^0_{\text{fram}}\co\mathcal{IP}^0/\sim \ \to\mathcal{M}_{\text{fram}}^{0}.
\end{align*}
\end{prop}

To prove Proposition \ref{framed branched spine is one to one}, we use the following lemmas.
\begin{lem}[{\cite{BP}*{Proposition 7.2.2}}]\label{BPlemma}
Let $P$ be a closed branched polyhedron.
For $x,x'\in C^1(P;\mathbb{Z})$ such that $\delta x=\delta x'=-c_P$, the two framing given by $x,x'$ coincide if and only if $[x_2-x_2^{\prime}]_2\in H^1(P; \mathbb{Z}_2)$ vanishes.
\end{lem}
\begin{lem}
\label{lem:0 mod 2 can be lifted}
Let $M$ be a closed oriented 3-manifold.
The following two conditions are equivalent.
\begin{enumerate}
    \item The kernel of $\text{mod}\,2\co H^1(M;\mathbb{Z}) \to H^1(M;\mathbb{Z}_2)$ is $0$.
    \item $b_1(M)=0$, where $b_1(M)$ is the first Betti number of $M$.
\end{enumerate}
\end{lem}
\begin{proof} Note that $H^1(M;\mathbb{Z})=\text{Hom}(H_1(M),\mathbb{Z})$ and thus $H^1(M;\mathbb{Z})$ is a free abelian group, which implies (2) $\Rightarrow$ (1). 
We prove (1) $\Rightarrow$ (2). 
Consider the following long exact sequence of cohomology 
\begin{align*}
    \cdots\xrightarrow[]{}H^1(M;\mathbb{Z})\xrightarrow[]{\times	2}H^1(M;\mathbb{Z})\xrightarrow[]{\text{mod}\,2}H^1(M;\mathbb{Z}_2)\xrightarrow[]{}\cdots,
\end{align*}
which is induced by the short exact sequence  $0\xrightarrow[]{}\mathbb{Z}\xrightarrow[]{}\mathbb{Z}\xrightarrow[]{}\mathbb{Z}_2\xrightarrow[]{}0$.
Then the condition (1) is equivalent to $\text{Tor}(\mathbb{Z}_2,H^1(M,\mathbb{Z}))=H^1(M,\mathbb{Z})$ by $\text{Im}(\times 2)=\text{Ker}(\text{mod}\,2)=0$.
Since $H^1(M;\mathbb{Z})$ is a free abelian group, we have $\text{Tor}(\mathbb{Z}_2,H^1(M;\mathbb{Z}))=0$, which implies $H^1(M;\mathbb{Z})=0$ and thus $b_1(M)=0$. 
\end{proof} 

\begin{proof}[Proof of Proposition \ref{framed branched spine is one to one}]
We need to prove the well-definedness and injectivity of the map $\Phi_{\text{fram}}$. The proof will be performed in almost the same way as in \cite{BP}*{Section 7.3}, where the authors use  $\mathbb{Z}/2\mathbb{Z}$ weights instead of integer weights.
The proof of well-definedness under the integral 0-2 move and the integral MP-move can be performed by comparing the second vector fields associated with the framing before and after the branched 0-2 and the branched MP-move as in Figure 7.7 and Figure 7.8 in \cite{BP}*{Section 7.3}. Note that the H-move adds a coboundary of $0$-chains, and thus the well-definedness is ensured by Lemma \ref{BPlemma}.
For the proof of injectivity, let $(P_1,x_1), (P_2,x_2)\in \mathcal{IP}^0$ be branched spines representing equivalent framed 3-manifolds $(M,v_1, v_2)$ and $(M^{\prime},v_1^{\prime}, v_2^{\prime})$, i.e., $\Phi_{\text{fram}}(P_1,x_1) = \Phi_{\text{fram}}(P_2,x_2)$.
Since $(M,v_1)$ and $(M^{\prime},v_1^{\prime})$ are equivalent closed oriented combed 3-manifolds, by Proposition \ref{prop:combed mfd and branched spine}, we can transform the spine $P_1$ to $P_2$ by the branched 0-2 move and the branched MP-move.
We perform the above transformation on $(P_1, x_1)$ using the integral 0-2 move and the integral MP-move instead of the branched 0-2 move and the branched MP-move, respectively, and denote $(P_2, x^{\prime}_2)$ the result.
Here, $x_2^{\prime}$ is not necessarily equal to $x_2$, but they can be transformed into each other by the H-move as follows.
Recall that $\delta x_2 = \delta x_2^{\prime}= -c_{P_2}$, thus we can define a cohomology class $[x_2-x_2^{\prime}]\in H^1(P_2;\mathbb{Z})$, which vanishes in $H^1(P_2;\mathbb{Z}_2)$ by Lemma \ref{BPlemma} since the two framings given by $x_2, x_2^{\prime}$ coincide.
Since $b_1(M)=0$, $[x_2-x_2^{\prime}] = 0$ in $H^1(P_2;\mathbb{Z})$ by Lemma \ref{lem:0 mod 2 can be lifted}, thus $x_2^{\prime}$ can be transformed to $x_2$ by adding the coboundary of 0-chains, which is performed by a sequence of the H-moves.
\end{proof}
\noindent
\textit{Proof of Proposition \ref{prop: framed 3-mfd BP-diagram}.}
The assertion follows from Proposition \ref{framed branched spine is one to one} and the identification of $\mathcal{IP}^0$ and  $\mathcal{FIG}^0$.
$\square$

\begin{rem}
    If we consider the mod 2 reduction of weights of a framed integral normal o-graph, the result is a framed normal o-graph defined in \cite{BP}*{Page 7}.
    Note that we do not need the integer lift of the weights to represent a closed framed 3-manifold because of the Lemma \ref{BPlemma}. However, if we consider the modulo $2$ reduction of the weights, it demands $S^4=\text{id}_H$ when we define the invariant using a Hopf monoid (Section \ref{sec:Invariant}). Since we would like to use  arbitrary non-involutory Hopf monoids such as small quantum Borel subalgebras, we prefer to use integer weights even though we add the condition $b_1=0$.
\end{rem}

\subsection{Examples of framed integral normal o-graph}
\label{Example of closed BP diagram}

We give some examples of framed integral normal o-graphs which represent lens spaces or branched covers of the trefoil knot.

\begin{ex}
\label{lens}
For $p\geq 1$, the following closed normal o-graph gives an example of a closed branched spine of the combed lens space $L(p,1)$:
\begin{figure}[H]
    \centering
    \includegraphics{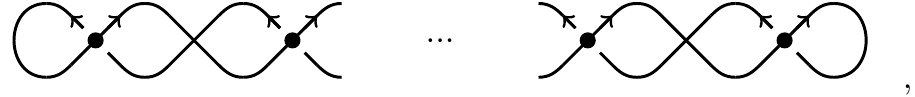}
\end{figure}
\noindent where the number of the vertices is $p$.
Especially for each $p=1,2$, the Euler class of this combing vanishes, thus we can extend this combing to a framing.
Since the first cohomologies of both manifolds are $0$, these extensions are actually canonical, and the framed integral normal o-graphs are given below:
\begin{figure}[H]
    \centering
    \includegraphics{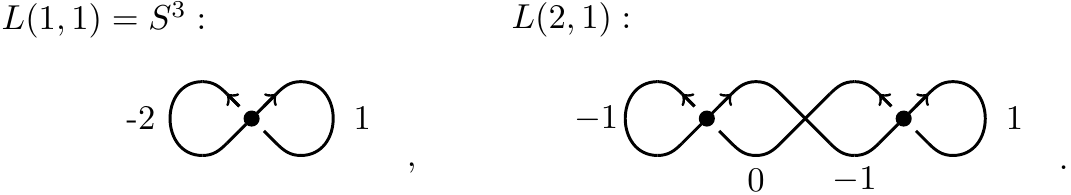}
\end{figure}
\noindent
In case of $p\geq 3$, the Euler class of the above combing of $L(p,1)$ does not vanish, thus we cannot extend the combing to framings. 
\end{ex}

\begin{ex}
For $p\geq 1$, let $M_p(K)$ be a $p$-fold cyclic branched cover of trefoil knot $K$ in $S^3$. 
 For $M_1(K)=S^3$, $M_2(K)=L(3,1)$, and $M_3(K)=S^3/Q_8$, we can construct the following framed integral normal o-graphs:
\begin{figure}[H]
    \centering
    \includegraphics{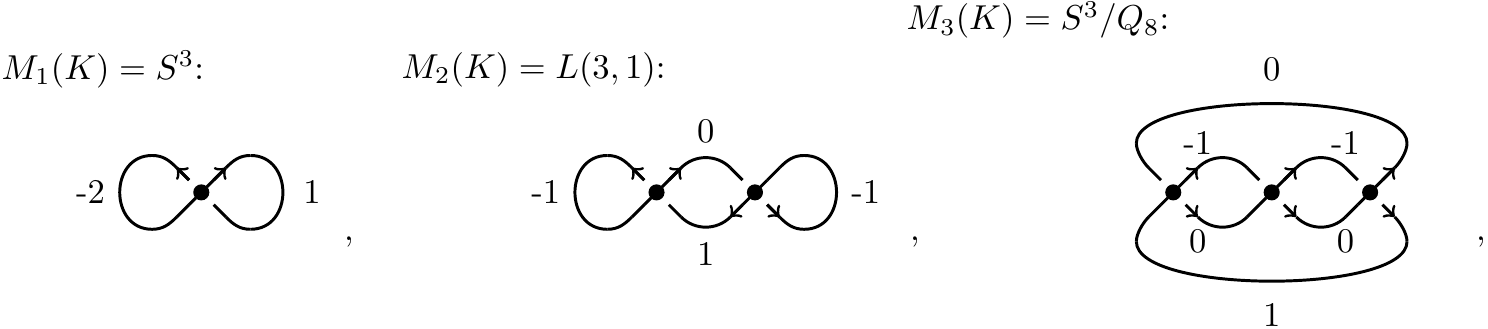}
\end{figure}
\end{ex}
\noindent which are obtained by reversing the constructions of DS-diagrams of $M_p(K)$ in  \cites{Hakone1}. \footnote{Let $P$ be a standard spine of $M$ with $S^2$ boundary. Since $M$ collapses to $P$, there is a map $\pi$ from $S^2\cong\partial M$ to $P$. The trivalent graph $\pi^{-1}(S(P))$ with vertices $\pi^{-1}(V(P))$ is called the DS-diagram of $(M,P)$.}

\section{Examples}
We give some explicit calculations of the invariant for a finite dimensional Hopf algebra $H$  over a field $\mathbb{K}$.

\subsection{Integral of Hopf algebra}\label{int}

Let $H$ be a finite dimensional Hopf algebra.
A \textit{\textbf{right integral}} $\mu_R\in H^*$ and a \textit{\textbf{right cointegral}} $e_R\in H$ are elements which satisfy
\begin{align*}
    \mu_R(x_{(1)})x_{(2)}=\mu_R(x)\cdot 1_H, \quad e_R x=\epsilon(x)e_R,
\end{align*}
respectively, for any $x\in H$.
It is known that each finite dimensional Hopf algebra admits unique non-zero integrals up to scalar multiplication.
In what follows we assume $\mu_R(e_R)=1$.

\begin{lem}\label{antipode-integral}
We have
\begin{figure}[H]
    \centering
    \includegraphics{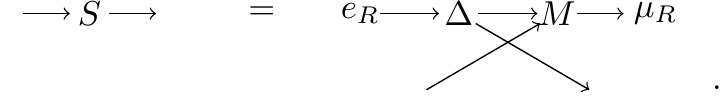}
\end{figure}
\end{lem}
\begin{proof}
The assertion follows from the uniqueness of the antipode for a finite dimensional Hopf algebra and from the fact that the above tensor network satisfies the axioms of antipode.
\end{proof}
Let $a=\mu_R(e_{R(2)})e_{R(1)}\in H$ and $\alpha=\mu_R(?e_R)\in H^*$, which are group-like elements, i.e., we have $\Delta(a)=a\otimes a$ and $\alpha(xy)=\alpha(x)\alpha(y)$. Set $q=\alpha(a)$. Note that $q$ does not depend on the choice of integrals satisfying $\mu_R(e_R)=1$.

Set $e_L= S^{-1}(e_R)$ and $\mu_L=\mu_R\circ S$, which turn out to be a \textit{\textbf{left cointegral}} and a \textit{\textbf{left integral}}, respectively.

\begin{lem}\label{lem:eigenvalue of antipode}
We have
\begin{align*}
    S(e_R) &= qe_L, \quad \mu_L(e_R)=q.
\end{align*}
\end{lem}
\begin{proof}
We prove the first equality. We can prove the second one similarly.

We set $S(e_R)=ke_L$ for $k\in \mathbb{K}$ and prove $k=q$. 
We have  $\mu_R(S^2(e_R))=\mu_R(kS(e_L))=k \mu_R(e_R)=k$ by the normalization,
thus it is enough to prove $\mu_R(S^2(e_R))=q$.
We have $\mu_R(S^2(e_R))=$
\begin{figure}[H]
    \centering
    \includegraphics{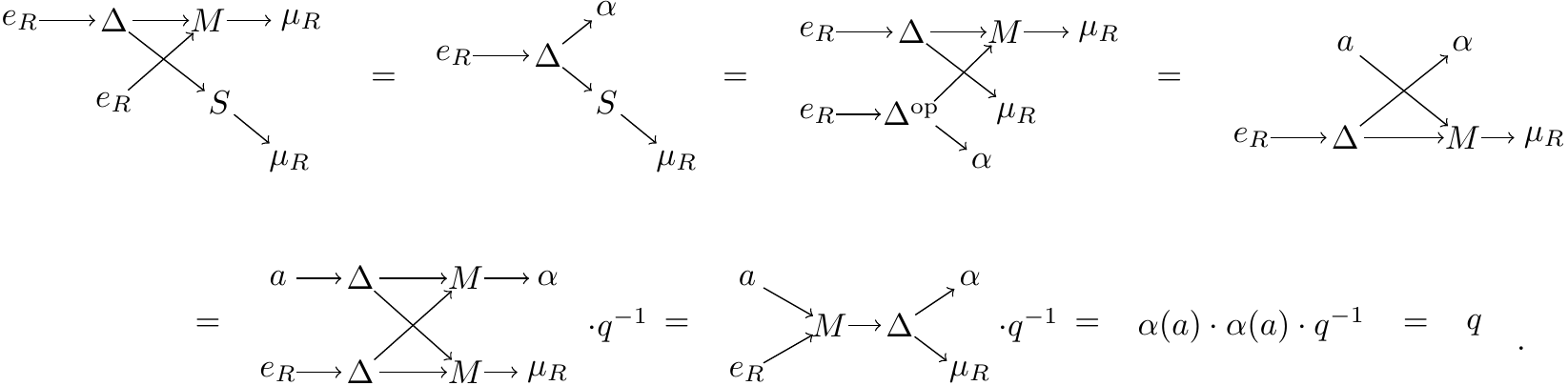}
\end{figure}

\noindent Here, the first and the third equalities are by Lemma \ref{antipode-integral}, the second and the fourth equalities follow from the definitions of $\alpha$ and $a$, respectively, and  the fifth equality follows from the fact that $a$ and $\alpha$  are group-like, and the sixth equality uses the definitions of $\alpha$ and $a$ at the same time.

\end{proof}

\begin{lem}\label{lem:twist lemma}
We have
\begin{figure}[H]
    \centering
    \includegraphics[scale=0.8]{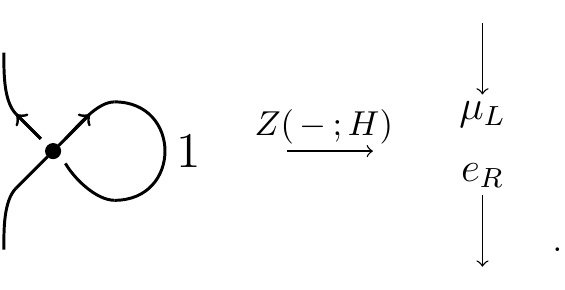}
\end{figure}
\end{lem}
\begin{proof}

We have
\begin{figure}[H]
    \centering
    \includegraphics{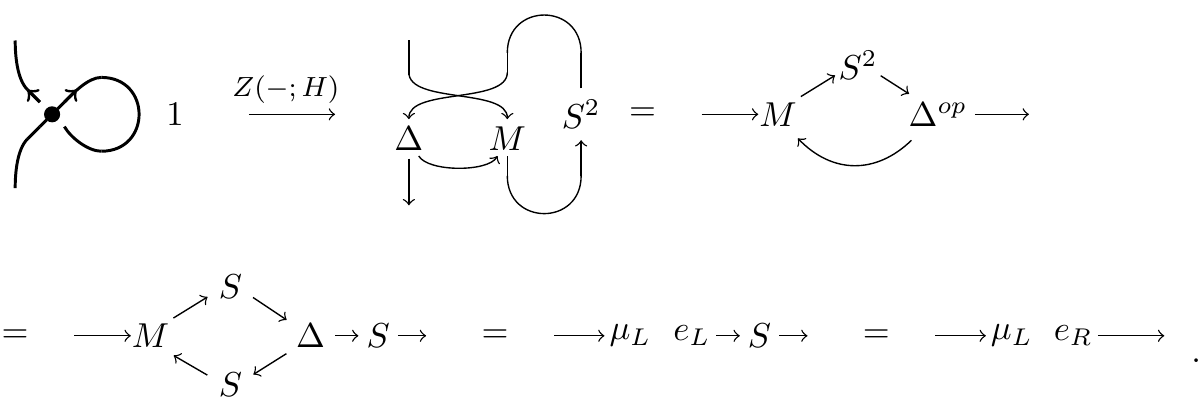}
\end{figure}
\noindent For the third equality, see \cite{Ku2}*{Lemma 3.3}.

\end{proof}
\subsection{Calculation  for $S^3$ and  $L(2,1)$ with certain framings}\label{subsec: Hopf ex}

\begin{ex}
\label{ex:invariant of S3}

For the framed $S^3$  in Example \ref{lens},
 by Lemma \ref{lem:twist lemma} we have
\begin{figure}[H]
    \centering
    \includegraphics{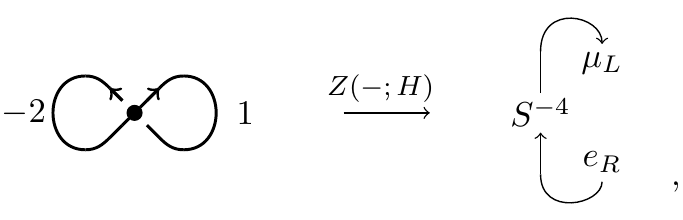}
\end{figure}

\noindent and then by Lemma  \ref{lem:eigenvalue of antipode} we have
\begin{align*}
Z(S^3,f;H)&=\mu_L(S^{-4}(e_R))
\\&=q^{-1}.    
\end{align*}
\end{ex}

\begin{ex}
For the framed lens space $L(2,1)$ in Example \ref{lens}, we have
\begin{figure}[H]
    \centering
    \includegraphics[scale=0.85]{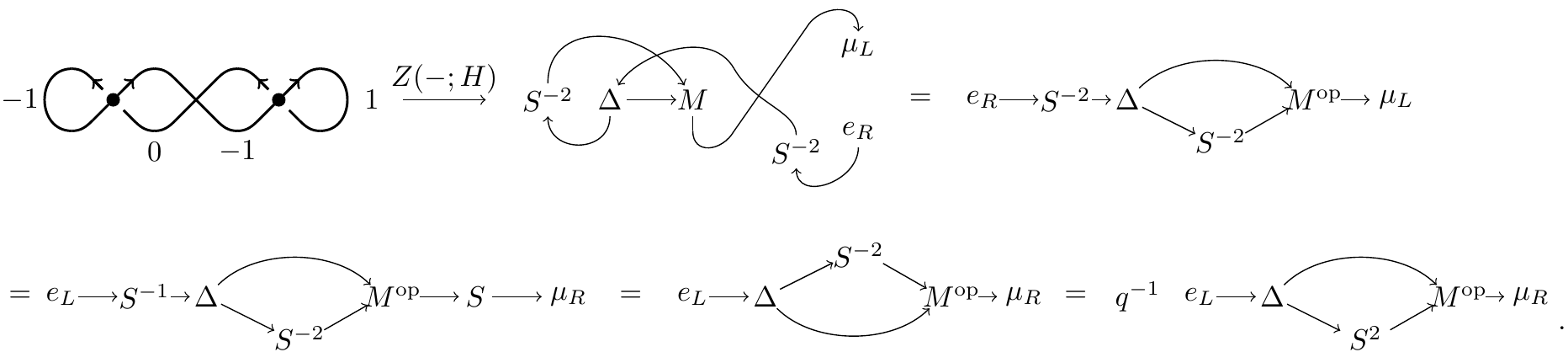}
\end{figure}

\noindent Since $\mu_R(e_L)=1$, by \cite{R}*{Theorem 10.4.1}, we have
\begin{align*}
    Z(L(2,1),f;H) &= q^{-1}\mu_R(S^2(e_{L(2)})e_{L(1)}) \\
    &= q^{-1}\text{Tr}(S).
\end{align*}
\end{ex}

\begin{rem}Both of the above examples match with the Kuperberg invariant \cite{Ku2}  up to multiplication by $q$. As we mentioned in the introduction,  we do not know the explicit relation in general between our invariant and the Kuperberg invariant.
\end{rem}

\subsection{Invariant for small quantum Borel subalgebra of $\mathfrak{sl}_2$ and $\text{SO}(3)$ WRT invariant}
Let us consider the invariant for the Hopf algebra $u_q(\mathfrak{sl}_2^+)$ (Example \ref{quantum Borel subalgebra}) with $q$ the $n$-th primitive root of unity.
In this case, the scalar $q=\alpha(a)$ defined in Section \ref{int} coincides, and thus we have
\begin{align*}
    Z(S^3,f;u_q(\mathfrak{sl}_2^+))=q^{-1}.
\end{align*}

In the case of $L(2,1)$, computing the $\text{Tr}(S)$ of $u_q(\mathfrak{sl}_2^+)$ in some basis such as the one in \cite{Ku2}*{Section 5} results in
\begin{align*}
    Z(L(2,1),f;u_q(\mathfrak{sl}_2^+))=2q^{-1}\frac{1-q^{-\lfloor \frac{n+1}{2} \rfloor}}{1-q^{-1}}.
\end{align*}

When $q$ is a primitive root of unity of odd order $N$, the above values times the cardinality of the first homology group matches with $\text{SO}(3)$ WRT invariant $\tau_N^{\text{SO}(3)}(M)$ defined in \cite{KM} up to multiplication by $q$.

\begin{conj}\label{vsWRT}
Let $q$ be a primitive root of unity of odd order $N$.
Then for every closed oriented framed 3-manifold $M$ with $b_1(M)=0$ there exists an integer $k$ such that
\begin{align*}
Z(M,f;u_q(\mathfrak{sl}_2^+))=q^k\cdot|H_1(M)|\cdot\tau_N^{\text{SO}(3)}(M)
\end{align*}
where $|H_1(M)|$ is the cardinality of the first homology group.
\end{conj}

Recall that WRT invariant is an invariant of 2-framed 3-manifold, where one usually chooses canonical 2-framing to compute it \cite{Ati}.
Since framing $f$ induces a 2-framing $v_2$, we expect that the following holds: 
\begin{align*}
    Z(M,f;u_q(\mathfrak{sl}_2^+))=|H_1(M)|\cdot \tau_N^{\text{SO}(3)}(M, v_2).
\end{align*}
\begin{appendices}
\section{Heisenberg double construction}
\label{Heisenberg double construction}

In this appendix, we show an alternative construction of the invariant $Z(-;H)$ based on the Heisenberg double of a Hopf algebra. 
In this construction, we can see that the invariance of $Z(-;H)$ under the integral MP-move  comes from the pentagon equation of the canonical element of the Heisenberg double. 
Although the following arguments hold also for any Hopf monoid, in order to make the appendix short, we assume $H$ to be a finite dimensional Hopf algebra over a field $\mathbb{K}$.

\subsection{Heisenberg double}
We use the left action of $H$ on $H^*$ defined by $(a\rightharpoonup f)(x):=f(xa)$, for $a, x\in H$ and $f\in H^*$.
  The \textit{\textbf{Heisenberg double}} of $H$ is a $\mathbb{K}$-algebra $\HD(H)=H^*\otimes H$ with the unit $\epsilon\otimes 1$ and the product given by
  \begin{align*}
      (f\otimes a)\cdot(g\otimes b)=f\cdot(a_{(1)}\rightharpoonup g)\otimes a_{(2)}b
  \end{align*}
for $f,g\in H^*$ and $a,b\in H$.

The Heisenberg double has a \textit{\textbf{canonical element}}
\begin{align*}
    T=\sum_{i}(\epsilon\otimes e_i)\otimes(e^i\otimes 1) \in \HD(H)\otimes \HD(H)
    \end{align*}
   with the inverse  
 given by 
    \begin{align*}
    \overline{T}=\sum_{i}(\epsilon\otimes S(e_i))\otimes(e^i\otimes 1) \in \HD(H)\otimes \HD(H).
\end{align*}
It is known  \cites{BS,Ka}  that $T$ satisfies the pentagon equation 
\begin{align*}
T_{12}T_{13}T_{23}=T_{23}T_{12}  \in \HD(H)^{\otimes 3},
\end{align*}
 where $T_{12}=T\otimes 1$ etc.

Set
\begin{align*}
    G=\sum_{i,j}e^ie^j\otimes S^{-1}(e_j)S^{2}(e_i) \in\HD(H).
\end{align*}
We call $G$ a \textit{\textbf{pivotal element}} since $G$ is an analog of a pivotal element in a Hopf algebra, i.e., we have
\begin{align*}
    Gx=\Theta^2(x)G, \quad x\in\HD(H)
\end{align*}
for $\Theta=S^*\otimes S^{-1}\co\HD(H)\to \HD(H)$ being an analog of the antipode.

Recall that $\HD(H)$ has the canonical left module  $F(H^*)=(H^*, \rho)$ called the \textit{\textbf{Fock space}}, where 
\begin{align*}
    \rho(f\otimes a)(h)=f\cdot(a\rightharpoonup h)
\end{align*}
for $f\otimes a\in\HD(H)$ and $h\in H^*$.
It is known that the Heisenberg double  is semisimple and the Fock space is its only simple module.
Furthermore, the character $\raisebox{2pt}{$\chi$}_{Fock} \co \HD(H)/[\HD(H),\HD(H)]\to\mathbb{K}$ associated to $F(H^*)$ is actually an isomorphism (cf. \cite{MST}), where $[\HD(H),\HD(H)]$ is the vector space over $\mathbb{K}$ spanned by $\{xy-yx\,|\,x,y\in\HD(H)\}$.

\subsection{Alternative definition of invariant}
An \textit{\textbf{$\HD(H)$-decorated diagram}} is an oriented closed curve immersed in $\mathbb{R}^2$, where the self-intersections are transverse double points, with a finite number of dots each of which is labeled with an element of $\HD(H)$.
These dots are called beads.
We consider $\HD(H)$-decorated diagrams up to planer isotopy and the moves in Figure \ref{fig:crm} and \ref{fig:bs}. Note that we can  slide beads freely along the curve using planer isotopy and beads sliding over a crossing.

\begin{figure}[H]
  \centering
  \begin{minipage}{0.61\columnwidth}
    \centering
    \includegraphics{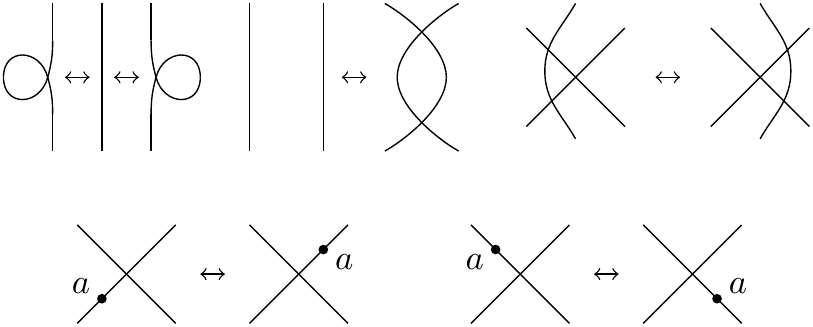}
    \caption{Reidemeister type moves and beads sliding over crossing.}
    \label{fig:crm}
  \end{minipage}
    \begin{minipage}{0.37\columnwidth}
    \centering
    \includegraphics{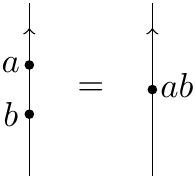}
    \caption{Product of beads.}
    \label{fig:bs}
  \end{minipage}
\end{figure}

Let $\Gamma$ be  an integral normal o-graph.  For simplicity, we assume that $\Gamma$ has one component as a knot diagram.  We define a scalar $Z(\Gamma;\HD(H))\in \mathbb{K}$ as follows. Let us write the canonical element $T$ as $T_1\otimes T_2$ and its inverse  $\overline{T}$ as $\overline{T}_1\otimes \overline{T}_2$ using Sweedler notation.
We replace each crossing and edge as in Figure \ref{fig:beads associated to vertices} to get a $\HD(H)$-decorated diagram associated with $\Gamma$. 
\begin{figure}[H]
  \includegraphics{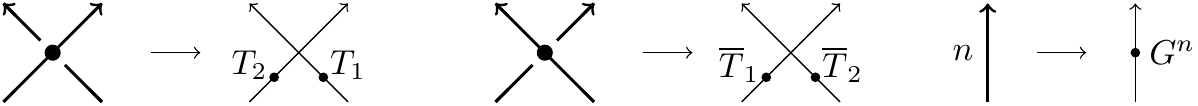}
  \caption{Definition of invariant.}
  \label{fig:beads associated to vertices}
 \end{figure}
\noindent Then, we slide the beads along the curves and merge them into one bead labeled by an element $J\in\HD(H)$, 
which is well-defined only in $\HD(H)/[\HD(H),\HD(H)]$
because of the relation in Figure \ref{fig:bs2}.
\begin{figure}[H]
  \centering
  \begin{minipage}{0.5\columnwidth}
    \centering
    \includegraphics{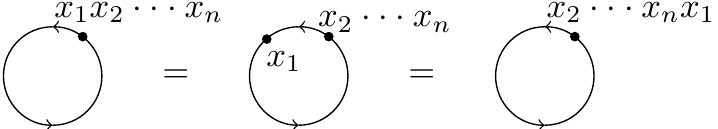}
    \caption{Cyclicity of beads.}
    \label{fig:bs2}
  \end{minipage}
\end{figure}
\noindent By composing the linear isomorphism $\raisebox{2pt}{$\chi$}_{Fock} \co \HD(H)/[\HD(H),\HD(H)]\to\mathbb{K}$, we get a well-defined scalar $Z(\Gamma;\HD(H))=\raisebox{2pt}{$\chi$}_{Fock}(J)
\in\mathbb{K}$. Thus we obtain a map $Z(-;\HD(H))\co \mathcal{IG} \to \mathbb{K}$, which indeed gives an invariant of integral normal o-graphs. 

Note that the integral MP-move is the Pachner $(2,3)$ move (with a certain framing inside) in the dual ideal triangulation, and 
the image of the integral MP-move by the invariant is equivalent to the pentagon equation; for example for the following integral MP-move:
\begin{figure}[H]
    \centering
    \includegraphics[scale=0.8]{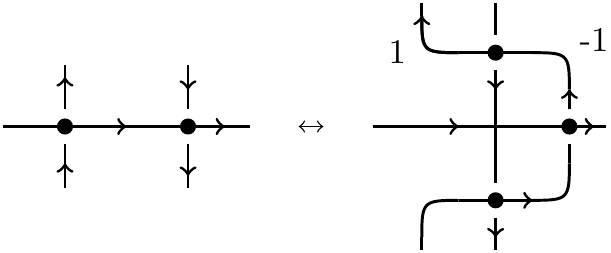}
\end{figure}
\noindent we have
\begin{figure}[H]
    \centering
    \includegraphics{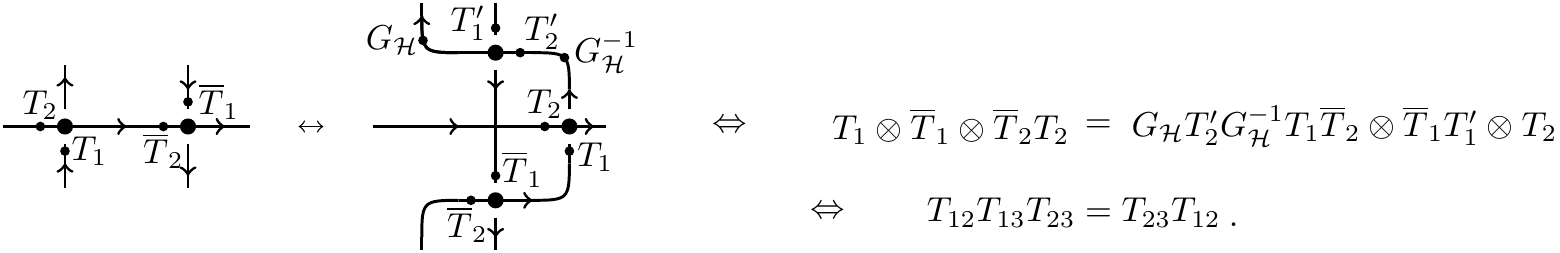}
\end{figure}
\noindent Observe that if we do not consider integer weights, i.e., if we consider the branched MP-move, then the resulting equality is a little different from the pentagon equation by the effect of $S^2$. The pivotal element plays a role to cancel out those effects, and the equality reduces to the pentagon equation.

Recall from Section \ref{sec:IF} the invariant $Z(-;H)$ of integral normal o-graphs.
\begin{prop}
Let $H$ be a finite dimensional Hopf algebra and $\Gamma$ an integral normal o-graph.
Then  we have
\begin{align*}
    Z(\Gamma;\HD(H))=Z(\Gamma;H).
\end{align*}
\end{prop}
\begin{proof}
We should check that the actions  of  $T$, $\overline{T}$, and $G$ on the Fock space $F(H^*)$ match with the tensor networks in Figure \ref{fig:Definition of the functor}.
For $T$ and $\overline{T}$, see the proof of \cite{MST}*{Proposition 6.2}. For $G$, we can check $\rho(G)(h)=(S^2)^*h$ for $h\in F(H^*)$ by a straightforward calculation.
\end{proof}

\end{appendices}

\end{document}